\documentclass[11pt,a4paper]{amsart}
\usepackage[english]{babel}
\usepackage{pst-grad} % For gradients
\usepackage{pst-plot} % For axes
\usepackage{pstricks}
\usepackage{amsmath,amssymb,mathrsfs,amsthm}
\usepackage[dvips]{graphicx}
\usepackage{epsfig}
\newcommand{\RR}{\mathbb R}
\newcommand{\Ss}{\mathcal S}

\newcommand{\NN}{\mathbb N}

\newcommand{\TT}{\mathbb T}
\newcommand{\pat}{\partial_t}
\newcommand{\pax}{\partial_x}

\newcommand{\vertiii}[1]{{\left\vert\kern-0.25ex\left\vert\kern-0.25ex\left\vert #1 
    \right\vert\kern-0.25ex\right\vert\kern-0.25ex\right\vert}}

\newcommand{\re}{\text{Re}}
\newcommand{\im}{\text{Im}}

\usepackage[pdfauthor={Rafael Granero-Belinchon},pdftitle={...},bookmarks,colorlinks]{hyperref}

\newcounter{comentcount}
\newcounter{teocount}
\newtheorem{lem}{Lemma}
\newtheorem{prop}{Proposition}

\newtheorem{teo}[teocount]{Theorem}  
\newtheorem{defi}{Definition}

\newtheorem{remark}{Remark}

\title[On the effect of boundaries in two-phase porous flow]{On the effect of boundaries in two-phase porous flow}
\author[R. Granero-Belinch\'{o}n]{Rafael Granero-Belinch\'{o}n}
\email{rgranero@math.ucdavis.edu}
\address{Department of Mathematics, University of California, Davis, CA 95616, USA}
\author[G. Navarro]{Gustavo Navarro}
\email{gnavarro@math.ucdavis.edu}
\address{Department of Mathematics, University of California, Davis, CA 95616, USA}
\author[A. Ortega]{Alejandro Ortega}
\email{alejandro.ortega@uc3m.es}
\address{Departmento de Matem\'aticas, Universidad Carlos III de Madrid, 28911, Legan\'es, Madrid, Spain}
\begin{document}

\begin{abstract}
In this paper we study a model of an interface between two fluids in a porous medium. For this model we prove several local and global well-posedness results and study some of its qualitative properties. We also provide numerics.
\end{abstract}

\maketitle



\textbf{Acknowledgments}: RGB and GN thanks Prof. Garving K. Luli for fruitful discussions. RGB and GN gratefully acknowledge the support by the Department of Mathematics at UC Davis where this research was performed. RGB is partially supported by the grant MTM2011-26696 from the former Ministerio de Ciencia e Innovaci\'on (MICINN, Spain).

\tableofcontents

\section{Introduction}
Free boundary problems for incompressible, inviscid flows and for active scalars are mathematically challenging and physically interesting. Moreover, their applications are really spread, from geothermal reservoirs (see \cite{CF}) to tumor growth (see \cite{F}), passing through weather forecasting (see \cite{majda1996two, constantin1994singular}).

In particular, the evolution of a fluid in a porous medium is important in the Applied Sciences and Engineering (see \cite{bear}) but also in Mathematics (see, for instance, \cite{c-c-g10}). The effect of the medium has important consequences and the usual equations for the conservation of momentum, \emph{i.e.} the Euler or Navier-Stokes equations, must be replaced with an empirical law: Darcy's Law
\begin{equation}
\frac{\mu}{\kappa}v=-\nabla p -g(0,\rho),
\label{eq1} 
\end{equation}
where $\mu$ is the dynamic viscosity of the fluid, $\kappa$ is the permeability of the porous medium, $g$ is the acceleration due to gravity, $v$ is the velocity of the fluid, $\rho$ is the density and $p$ is the pressure (see \cite{bear}). In our favourite units, we can assume $g=\mu=1.$

A very important part of the theory of flow in porous media studies the coexistence of two immiscible fluids with different qualities in the same volume. The case of two immiscible and incompressible fluids is known as the Muskat o Muskat-Leverett problem (see \cite{Muskat} and also \cite{SCH}). In this case the density is given by 
\begin{equation}\label{density}
\rho=\rho^2\textbf{1}_{\{y<f(x,t)\}}+\rho^1\textbf{1}_{\{y>f(x,t)\}},
\end{equation}
where 
\begin{equation}\label{gamma}
\Gamma(t)=\{(x,f(x,t)): \: x\in\RR\},
\end{equation} 
is the interface between both phases. This interface is an unknown in the evolution. If $\rho^2>\rho^1$, the system is in the so-called \emph{stable} (or Rayleigh-Taylor stable) regime.

Given the depth of the porous medium, $l>0$, we define the following dimensionless parameter (see \cite{bona2008asymptotic} and references therein)
\begin{equation}\label{dimensionless}
\mathcal{A}=\frac{\|f_0\|_{L^\infty}}{l}. 
\end{equation}

If the porous medium has infinite depth ($\mathcal{A}=0$), the equation for the interface is
\begin{equation}\label{full}
\pat f=\frac{\rho^2-\rho^1}{2\pi}\text{P.V.}\int_\RR \frac{(\pax f(x)-\pax f(x-\eta))\eta}{\eta^2+(f(x)-f(x-\eta))^2}d\eta,
\end{equation}
where $\text{P.V.}$ denotes principal value. This situation is known as the \emph{deep water regime}. This case has been extensively studied (see \cite{ambrose2004well, castro2012breakdown, ccfgl, ccgs-10, ccgs-13, c-c-g10, c-g07, c-g09, escher2011generalized, e-m10, KK} and references therein).

If the initial data verifies $|f_0|<l$, in the presence of impervious boundaries (see Figure \ref{bandafig}), the system is in the regime $0<\mathcal{A}<1$. This is known as the \emph{confined Muskat problem}. The equation for the interface corresponding to this situation when $l=\pi/2$ is
\begin{multline}
\pat f(x,t) = \frac{\rho^2-\rho^1}{4\pi}\text{P.V.}\int_\RR\bigg{[}\frac{ (\partial_x f (x)-\partial_x f (x-\eta ))\sinh\left(\eta\right)}{\cosh \left(\eta\right)-\cos\left(f(x)-f(x-\eta)\right)}\\
+\frac{(\partial_x f(x)+\partial_x f(x-\eta))\sinh\left(\eta\right)}{\cosh \left(\eta\right)+\cos\left(f(x)+f(x-\eta)\right)}\bigg{]}d\eta.
\label{eq0.1}
\end{multline}
This case has been studied in \cite{BCG, CGO, GG, G}. Notice that the second kernel becomes singular when $f$ reaches the boundaries.

\begin{figure}[h!]
		\begin{center}
		\includegraphics[scale=0.3]{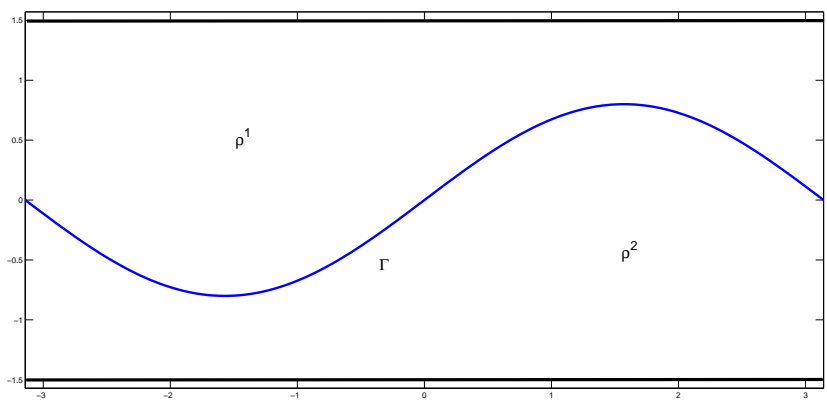} 
		\end{center}
		\caption{Physical situation for an interface $\Gamma$ in the strip $\RR\times (-l,l)$.}
\label{bandafig}
\end{figure}

It has been proved that $\mathcal{A}$ plays an important role on the evolution of $\|f(t)\|_{C^1}$. For instance, if $0<\mathcal{A}<1$, $\|f(t)\|_{L^\infty}$ decays slower than in the deep water regime. Moreover, to ensure that $\|\pax f(t)\|_{L^\infty}<1$ for every time, one needs to impose conditions on the amplitude and the slope of the initial data related to the depth $l$ (see \cite{CGO}). Notice that in the deep water case, the condition is only on $\|\pax f_0\|_{L^\infty}$ (see \cite{c-g09}). Finally, we also mention that, using a computer assisted proof, the authors in \cite{GG} proved that there exists curves, $z=(z_1(\alpha),z_2(\alpha))$, solutions of of the confined Muskat problem ($0<\mathcal{A}<1$) corresponding to the initial data $z_0$, such that $\pax z_1(0,\delta)<0$ for a sufficiently small time $\delta>0$ (\emph{i.e.} the wave \emph{breaks}). The same initial curve $z_0$ when plugged into the infinitely deep Muskat problem ($\mathcal{A}=0$) verifies $\pax z_1(0,\delta)>0$ (\emph{i.e.} the wave becomes a smooth graph).

The last case corresponds to $\mathcal{A}=1$. This case is known as the \emph{large amplitude regime}. In this situation, the initial interface reach the impervious walls at least in one point.

Let's write $\Lambda u=\sqrt{-\pax^2}u$ for the square root of the laplacian. In \cite{c-g-o08}, the authors proposed the problem
\begin{equation}\label{modelodiego}
\partial_t f(x)=-\frac{1}{1+(\partial_x f(x))^2}\Lambda f(x),
\end{equation}
as a model of the dynamics of an interface in the two phase, deep water Muskat problem \eqref{full}. If we define $g = \partial_x f$ and we take the derivative $\partial_x$ to equation \eqref{modelodiego}, we get
\begin{equation}\label{modelodiegoderiv}
g_t = -\Lambda g + \pax\left(\frac{g^2Hg}{1+g^2}\right).
\end{equation}
This latter equation will be helpful because it has divergence form.

In this paper, to study the effect of the boundaries in the large amplitude regime, we propose and study a model of \eqref{eq0.1}. Assume $l=\pi/2$ and $f(\tilde{x})=\pm l$ for some $\tilde{x}$. Then, the second term in \eqref{eq0.1} reduces to
$$
\frac{(\partial_x f(\tilde{x})+\partial_x f(\tilde{x}-\eta))\sinh\left(\eta\right)}{\cosh \left(\eta\right)+\cos\left(f(\tilde{x})+f(\tilde{x}-\eta)\right)}=\frac{\partial_x f(\tilde{x}-\eta)\sinh\left(\eta\right)}{\cosh \left(\eta\right)-\cos\left(f(\tilde{x})-f(\tilde{x}-\eta)\right)}\approx H\pax f(\tilde{x}),
$$
where $H$ denotes the Hilbert transform. Therefore, using \eqref{eq0.1}, $\pat f(\tilde{x})=0$ and the diffusion degenerates. 

To capture this crucial fact when $f(\tilde{x})=\pm l$, we introduce the equation 
\begin{equation}\label{model1}
\partial_t f(x)=\left(-\frac{1}{1+(\partial_x f(x))^2}+\frac{1}{1+l^2-\left(f(x)\right)^2}\right)\Lambda f(x),
\end{equation}
as a model of \eqref{eq0.1} in the large amplitude regime. Let's point out that, if $l=\infty$, formally, we recover \eqref{modelodiego}.

Notice that 
\begin{equation}\label{eqRayleigh}
\frac{1}{1+(\pax f(x))^2}>\frac{1}{1+l^2-\left(f(x)\right)^2},
\end{equation}
guarantees that the model \eqref{model1} is in the so-called stable regime. In other words, if \eqref{eqRayleigh} holds, the model has a nonlocal, non-degenerate diffusion. Notice that the set of functions verifying \eqref{eqRayleigh} is not empty.

To bound the stability condition, we define 
\begin{equation}\label{sigma}
\sigma(t)=\max_x\left\{\frac{1}{1+l^2-\left(f(x)\right)^2}-\frac{1}{1+\left(\pax f(x)\right)^2}\right\}.
\end{equation}
Using this function, the stability condition is equivalent to $\sigma(t)<0$.

\begin{remark}
Let us mention that 
$$
f(x)=\pm l\sin(x),\pm l\cos(x),
$$
are steady solutions in the unstable case.
\end{remark}

\subsection{Notation and functional framework}
We write $Hu$ for the Hilbert transform of the function $u$ and $\Lambda u=\sqrt{-\pax^2}u=\pax Hu$ for the Zygmund operator, \emph{i.e.}
$$
\widehat{Hu}(\xi)=-i\frac{\xi}{|\xi|}\hat{u}(\xi)\text{ and }\widehat{\Lambda u}(\xi)=|\xi|\hat{u}(\xi),
$$
where $\hat{\cdot}$ denotes the usual Fourier transform. 

We write $\Omega$ for our spatial domain. From this point onwards, we consider either $\Omega=\RR$ or $\Omega=\TT$, in particular our domain is always onedimensional. We write $H^s(\Omega)$ for the usual $L^2$-based Sobolev spaces with norm
$$
\|f\|_{H^s}^2=\|f\|_{L^2}^2+\|f\|_{\dot{H}^s}^2, \|f\|_{\dot{H}^s}=\|\Lambda^s f\|_{L^2},
$$
We denote $l>0$ the depth and
$$
H^s_l=\left\{f\in H^s\text{ s.t. }\|f\|_{L^\infty}< l\right\}.
$$
The fractional $L^p$-based Sobolev spaces, $W^{s,p}(\Omega)$, are 
$$
W^{s,p}=\left\{f\in L^p(\Omega), \pax^{\lfloor s\rfloor} f\in L^p, \frac{|\pax^{\lfloor s\rfloor}f(x)-\pax^{\lfloor s\rfloor}f(y)|}{|x-y|^{\frac{1}{p}+(s-\lfloor s\rfloor)}}\in L^p(\Omega\times\Omega)\right\},
$$
with norm
$$
\|f\|_{W^{s,p}}^p=\|f\|_{L^p}^p+\|f\|_{\dot{W}^{s,p}}^p, 
$$
$$
\|f\|_{\dot{W}^{s,p}}^p=\|\pax^{\lfloor s\rfloor} f\|^p_{L^p}+\int_\Omega\int_\Omega\frac{|\pax^{\lfloor s\rfloor}f(x)-\pax^{\lfloor s\rfloor}f(y)|^p}{|x-y|^{1+(s-\lfloor s \rfloor)p}}.
$$
Notice that $W^{k+s,\infty}$, $k\in\NN$, $1>s\geq0$, reduces to the usual H\"older continuous space $C^{k+s}$. We define
$$
\mathcal{C}_l=\frac{l^2}{1+l^2}. 
$$
This is the constant appearing in the linear problem. 

\subsection{Statement of the results for \eqref{model1}}
In this section we collect the statement of the results concerning the equation \eqref{model1}. We start with local well-posedness of classical solutions and a continuation criterion when the initial data is in the stable regime:
\begin{teo}\label{teo1}
Let $f_0\in H^s_l(\Omega)$, $s\geq3$, $l>0$, $\epsilon>0$ and $\Omega=\TT,\RR$, be an initial data satisfying the stability condition \eqref{eqRayleigh}. Then there exists a unique solution $f(x,t)$ to \eqref{model1} such that
$$
f(x,t)\in C([0,T(f_0)],H^3)\cap L^2([0,T(f_0)],H^{3.5}).
$$ 
Moreover, if $T^*$ is the maximum lifespan of the solutions, then, $T^*=\infty$ or
\begin{equation}\label{continuationcriteria}
\sup_{0\leq t\leq {T^*}}\|f(t)\|_{W^{2+\epsilon,\infty}}=\infty.
\end{equation}
\end{teo}

Notice that the continuation criteria \eqref{continuationcriteria}, as is written in the previous result, doesn’t deal with the possibility of reaching the unstable regime. However, in Proposition \ref{prop2} below we address this question.

We also prove that there is a unique local smooth solution even when the Rayleigh-Taylor condition \eqref{eqRayleigh} is not satisfied but our initial data is analytic. We prove this result complexifying the equation and using a Cauchy-Kowalevski Theorem (see \cite{nirenberg1972abstract} and \cite{nishida1977note}).

We define the complex strip $\Ss_r=\{x+i\xi,|\xi|<r\},$ and $\gamma=x\pm ir$, $\gamma'=x\pm ir'$ for $x\in\Omega$. We consider the Hardy-Sobolev spaces (see \cite{bakan2007hardy})
\begin{equation}\label{Xdef}
H^3(\Ss_r)=\{f(x+i\xi) \text{ analytic on } \Ss_r \text{ s.t. }f(x\pm ir)\in H^3(\Omega) \text{ and } f(x)\in\RR\},
\end{equation}
with norm
$$
\|f(\gamma)\|^2_{H^3(\Ss_r)}=\sum_{\xi=\pm r}\|f(\gamma)\|^2_{H^3(\Omega)}.
$$
These spaces form a Banach scale with respect to the parameter $r$. In the same way we define $\|f\|^2_{L^2(\Ss_r)}=\sum_{\xi=\pm r}\|f(\gamma)\|^2_{L^2(\Omega)}.$ We also have, for $0<r'<r$, 
\begin{equation}\label{cauchy2}
\|\partial_x \cdot\|_{L^2(\Ss_{r'})}\leq\frac{C}{r-r'}\|\cdot\|_{L^2(\Ss_r)}.
\end{equation}
The complex extension of the equation can be written as 
\begin{equation}\label{eqmodelIA}
\partial_t f(\gamma)=\left(-\frac{1}{1+(\partial_x f(\gamma))^2}+\frac{1}{1+l^2-\left(f(\gamma)\right)^2}\right)\Lambda f(\gamma),
\end{equation}
where
$$
\Lambda u(\gamma)=\frac{1}{\pi}\text{P.V.}\int_\RR\frac{u(\gamma)-u(\gamma-\eta)}{\eta^2}d\eta,\text{ }
\frac{1}{2\pi}\text{P.V.}\int_\TT\frac{u(\gamma)-u(\gamma-\eta)}{\sin^2\left(\frac{\eta}{2}\right)}d\eta.
$$
Notice that the variable $\eta$ is a real number: $\eta\in\Omega$. Given a positive $\tau<1$, we define 
\begin{equation*}
d_1[f](\gamma)=\frac{1}{\tau+l^2-\left(\re f(\gamma)\right)^2}, d_2[f](\gamma)=\frac{1}{\tau-\left(\im \pax f(\gamma)\right)^2},
\end{equation*}
Given $R>0$, we define the open set
\begin{equation*}
\mathcal{O}^\tau_R=\{f\in H^3(\Ss_r), \text{ s.t. }0<d_1[f]<R,0<d_2[f]<R, \|f\|_{H^3(\Ss_r)}<R\}.
\end{equation*}
We remark that in this set we have
$$
(\re f(\gamma))^2<l^2+\tau, (\im \pax f(\gamma))^2<\tau.
$$

\begin{teo}\label{teoCK}Let $f_0\in \mathcal{O}^{\tau_0}_{R_0}$ for some $0<\tau_0<1$, $R_0>0$ be the initial data for \eqref{model1}. Then, there exists $T(f_0)$ and a unique solution $f(x,t)\in C([-T,T],H^3(\RR))$. 
\end{teo}

This result is interesting because there exist functions such that $f(\tilde{x})=l$ in this set $O^\tau_R$. For instance one can consider $f^a_0(x)=a\cos(x)+l-a$ for a small enough $a$. In particular, this case is analogous to the case where the initial data reaches the boundary, \emph{i.e.} the large amplitude regime.

We study the decay of some lower order norms and other qualitative properties:
\begin{prop}\label{prop1}
Given $f_0\in H^3_l(\Omega),$ $l>0$, $\Omega=\RR,\TT$, in the stable regime, then the solution of \eqref{model1} verifies:
\begin{itemize}
\item 
The even/odd symmetry of the initial data propagates.
\item 
$$
\|f(t)\|_{L^\infty(\Omega)}\leq\|f_0\|_{L^\infty(\Omega)},
$$
\item Assume that $f_0\in H^3_l(\TT)$ is odd, then 
\begin{equation}\label{decay}
\|f(t)\|_{L^\infty(\TT)}\leq \frac{e^{-Ct}\|f_0\|_{L^\infty(\TT)}}{1+\mathcal{A}\left(e^{-Ct}-1\right)},
\end{equation}
where $\mathcal{A}$ is defined in \eqref{dimensionless}.
\item As long as the solution remains in the stable case, the solution is in $L^\infty_t\dot{H}^{0.5}_x\cap L^2_t\dot{H}^{1}_x$ and we have the following energy balance
$$
\|f(t)\|^2_{\dot{H}^{0.5}(\Omega)}-2\int_0^t \sigma(s)\|f(s)\|^2_{\dot{H}^{1}(\Omega)}ds\leq \|f_0\|^2_{\dot{H}^{0.5}(\Omega)},
$$
where $\sigma(t)<0$ is defined in \eqref{sigma}. Moreover, if the solution is in the stable regime up to time $T$, then
$$
\sup_{0\leq t\leq T}\|f(t)\|_{L^2}\leq c(\|f_0\|_{H^{0.5}},T).
$$ 

\item Assume that $f_0\in H^3_l(\TT)$, then 
\begin{equation}\label{decayh0.5}
\|f(t)\|_{\dot{H}^{0.5}(\TT)}\leq e^{2\int_0^t\sigma(s)ds}\|f_0\|_{\dot{H}^{0.5}(\TT)}.
\end{equation}
\item Given $f_0\in H^3_l(\RR)$ and assuming that the solution is in the stable regime in the time interval $[0,T]$, we obtain
$$
\|f(t)\|_{L^\infty(\RR)}\leq \left(\left(\frac{\frac{1-\mathcal{A}}{\mathcal{A}}}{1+l^2}\right)\frac{3t}{\mathfrak{C}(f_0,T)}+\frac{1}{\|f_0\|^3_{L^\infty(\RR)}}\right)^{-1/3},
$$
where 
$$
\mathfrak{C}(f_0,T)=\|f_0\|_{L^2}+ \frac{2\sqrt{T}\|f_0\|_{\dot{H}^{0.5}}}{\min_{0\leq s\leq T}\sqrt{|\sigma(s)|}},
$$ 
is a bound for $\|f(t)\|_{L^2(\RR)}$.

\end{itemize}
\end{prop}

Recall that \eqref{decay} in Proposition \ref{prop1} gives us that, in the case 
$$
\mathcal{A}=\frac{2}{\pi}\|f_0\|_{L^\infty(\TT)}\approx 1,
$$
(the interface is close to the boundary)
$$
\|f(t)\|_{L^\infty(\TT)}\leq \frac{e^{-\frac{\mathcal{C}_l}{l}t}\|f_0\|_{L^\infty(\TT)}}{1+\mathcal{A}\left(e^{-\frac{\mathcal{C}_l}{l}t}-1\right)}\approx \|f_0\|_{L^\infty(\TT)},
$$
and our decay estimate degenerates. This fact has been observed for equation \eqref{eq0.1} in \cite{CGO}. Moreover, it has also been observed in the numerical simulations in Section \ref{secdec} (see Figure \ref{deccase1}).

Notice that there is not a $L^2$ maximum principle, but we can use backward bootstrapping to bound the $L^2$ norm once that we now a bound for $\dot{H}^{0.5}$.

We prove that if the initial data is small, then there exists a global-in-time solution. Furthermore, we obtain some decay estimates in a lower norm. Thus, these results complement the decay rates proved in Proposition \ref{prop1}. We will use the approach in \cite{cheng2012global, shatahnormalforms}. Notice that, given $f_0\in H^3$, there exists a time of existence $T_0=T(f_0)$ and the solution is on the stable regime. For any $T<T_0$, we define the total norm 
\begin{equation}\label{totalnorm}
\vertiii{f}_T=\sup_{0<t<T}\{\|f(t)\|_{X}+\mathcal{D}(t)\|f(t)\|_{Y}\},
\end{equation}
where $X\subset Y$ are Banach spaces. The function $\mathcal{D}(t)\rightarrow\infty$ as $t\rightarrow\infty$ and gives us the decay in the lower order norm. Using Duhamel's principle we write the expression for the mild solution
\begin{equation}\label{Duhamel}
f(x,t)=e^{-t\mathcal{C}_l\Lambda}f_0+\int_0^te^{-(t-s)\mathcal{C}_l\Lambda}NL(s)ds,
\end{equation}
where
\begin{equation}\label{nonlinear}
NL=\Lambda f\left[\frac{(\pax f)^2\left(1+2l^2+l^4-l^2f^2\right)+f^2}{(1+(\pax f)^2)(1+l^2-f^2)(1+l^2)}\right].
\end{equation}
\begin{teo}\label{teo3}
Let $f_0\in H^3(\TT)$ be an odd initial data for equation \eqref{model1} in the stable regime.  Then, there exist $\delta>0$, such that if $\|f_0\|_{H^3(\TT)}\leq \delta$ the corresponding solution is global in time and the solution verifies
$$
\|f(t)\|_{C^1(\TT)}\leq \frac{l}{4} e^{-\mathcal{C}_lt}.
$$
\end{teo}

\begin{remark}
The oddness assumption is related to the decay estimate. We know that the odd character of the initial data propagates, so the solution will have zero mean and then the equilibrium solution is $f_\infty\equiv0$. However, as the mean is not preserved, it is not clear, and in general it is not true, that the mean will propagate for general initial data with zero mean.
\end{remark}

There are several results with limited regularity for \eqref{full} (see \cite{ccgs-13}). In particular, the authors in this paper proved the global existence of smooth solution corresponding to initial data with small derivative in the Wiener algebra. We prove that \eqref{model1} also captures these features. In particular, we study the equation \eqref{model1} when the initial data is only $H^2$ and we prove local existence for small initial data in both spatial domains, the real line and the torus. 
\begin{teo}\label{teo4}
Let $f_0\in H^2(\Omega)$, $\Omega=\TT,\RR$, be the initial data for equation \eqref{model1} in the stable regime. We assume that $\|f_0\|_{H^{2}(\Omega)}\leq \delta$ for a small enough $\delta>0$. Then, there exists at least one local solution
$$
f\in C([0,T(f_0)],H^2(\Omega))\cap L^2([0,T(f_0)],H^{2.5}(\Omega)).
$$
\end{teo}
Notice that the solution is classical, but if the initial data is only $H^2$ the well-posedness for arbitrary data can not be achieved by standard energy methods. In the case where the initial data is odd and periodic, we can improve the previous local-in-time result:

\begin{teo}\label{teo5}
Let $f_0\in H^2(\TT)$ be an odd initial data for equation \eqref{model1} in the stable regime.  Then, there exist $\delta>0$, such that if $\|f_0\|_{H^{2}(\TT)}\leq \delta$ there exists at least one global in time solution. This solution verifies
$$
\|f(t)\|_{H^1(\TT)}\leq \frac{l}{4} e^{-\mathcal{C}_lt},\text{ and }\|f(t)\|_{C^1(\TT)}\leq\|f_0\|_{C^1(\TT)}.
$$
\end{teo}

The two main possibilities for finite time blow up seem to be 
\begin{enumerate}
\item To reach the unstable regime,
\item a blow up of the curvature for the case $f(\tilde{x})=l$.
\end{enumerate}

To reach the unstable regime is similar to the \emph{turning} singularities presents for \eqref{full} and \eqref{eq0.1} in \cite{ccfgl} and \cite{BCG, CGO, GG}. We discard this situation for \eqref{model1}. In particular we prove that, if the solution reaches the unstable case, the $W^{2+\epsilon,\infty}$ blows up first. The second source of singularity, a blow up of the curvature when the initial data reaches the boundaries may take two different forms: a corner-type singularity (blow up of the second derivative while the first derivative is bounded) and a cusp-type singularity (blow up of the first and second derivatives). We prove that, if the second derivative blows up, then the norm $W^{1+\epsilon,\infty}$ blows up first. Notice that, as a consequence of our proof, we get that if the initial data reaches the boundary, then the solution corresponding to this initial data reaches the boundary as long as it remains smooth.

We collect these two results in the next proposition:

\begin{prop}\label{prop2}
Let $f_0\in H^3(\Omega)$ be the initial data for equation \eqref{model1} and $0<T<\infty$ be an arbitrary parameter. We assume that the corresponding solution is $f(x,t)\in C([0,T),H^3(\Omega))$. Then,
\begin{itemize}
\item If $f_0$ is in the stable regime and $T$ is the first time where the solution leaves the stable regime, then 
$$
\limsup_{t\rightarrow T}\int_0^t\|f(s)\|^2_{\dot{W}^{2+\epsilon,\infty}}ds=\infty.
$$ 
\item If $f_0$ is analytic and there exists $\tilde{x}$ such that $f_0(\tilde{x})=l$, then
$$
f(\tilde{x},t)=l,\,\forall 0\leq t\leq T
$$
and 
$$
\max_{0\leq t\leq T}|\pax^2 f(\tilde{x},t)|\leq c(l,f_0)e^{l\int_0^T\Lambda f(\tilde{x},s)ds}.
$$
Consequently, the curvature can not blow up for a $W^{1+\epsilon,\infty}$ solution.
\end{itemize}
\end{prop}

There are three main questions that remain open for this model: an existence theory for initial data in $H^{1.5}\cap W^{1,\infty}$, a proof of finite time singularities where the curvature blows up and the existence of a geometry (instead of a flat strip) that enhances the similarities between the Muskat problem and the model introduced in this paper.  

\subsection{Statement of the results for \eqref{modelodiego}}
We obtain a new energy balance for \eqref{modelodiego}. To do that, we consider the evolution of the entropy 
$$
\int_\Omega \pax f(t)\log(1+(\pax f(t))^2)dx.
$$
\begin{prop}\label{prop3}
Given $f_0\in W^{1,3}(\Omega)\cap W^{1,1}(\Omega),$ $\Omega=\RR,\TT$, then the solution of \eqref{modelodiego} verifies the following energy balance
\begin{multline}\label{energybalance}
\int_\Omega \pax f(t)\log(1+(\pax f(t))^2)dx\\
+2\int_\Omega \arctan(\pax f(t))-\int_0^t\int_\Omega   \frac{\Lambda \pax f}{1+(\pax f)^2}dxds\\
=\int_\Omega \pax f_0\log(1+(\pax f_0)^2)dx+2\int_\Omega \arctan(\pax f_0).
\end{multline}
\end{prop}

Furthermore, under a positiveness hypothesis for $\pax f_0$, we can use this energy balance to obtain global existence of weak solutions with rough initial data. This energy balance fully exploits the diffusive character of the equation \eqref{modelodiego}. Notice also that, due to the positiveness of $g$, we can not recover a smooth, periodic $f$ from this $g$. Now we define our notion of weak solution:

\begin{defi} \label{definition 1}
$g(x,t)$ is a global weak solution of \eqref{modelodiegoderiv} if 
$$
g(x,t)\in L^\infty([0,T],L^\infty(\TT))\cap L^2([0,T],\dot{H}^{0.5}(\TT))
$$ 
and \eqref{modelodiegoderiv} holds in the sense of distributions: for any $\psi\in C^{\infty}\left([0,T)\times\mathbb{T}\right)$, periodic in space and with compact support in time, 
\[
\int^{T}_{0}\int_{\mathbb{T}} -\Lambda\psi g+\pat\psi g - \pax\psi\left(\frac{g^2Hg}{1+g^2}\right)dxdt-\int_\TT \psi(x,0)g_0(x)dx=0, 
\]
for every $T<\infty$.
\end{defi} 

We state now our result:
\begin{teo}\label{teoweak}
Let $g_0\in L^\infty(\TT)$ be a positive initial data for equation \eqref{modelodiegoderiv}. Then, there exist at least one global weak solution
$$
g(x,t)\in L^\infty([0,T],L^\infty(\TT))\cap L^2([0,T],\dot{H}^{0.5}(\TT)),
$$ 
satisfying the bounds 
$$
\|g(t)\|_{L^\infty(\TT)}\leq\|g_0\|_{L^\infty(\TT)},\;\min_x g(x,t)>\min_x g_0(x)\;\forall\,t>0,
$$
and
$$
\int_0^t\|g(s)\|_{\dot{H}^{0.5}}^2\leq\frac{(1+\|g_0\|_{L^\infty})^2}{\min_x g_0}\left(\int_\TT g_0\log(1+g_0^2)dx+2\int_\TT \arctan(g_0)\right).
$$
\end{teo}

\begin{remark}
Notice that if we study the evolution of the energy
$$
\int (1 + \frac{1}{6}(\partial_x f (x, t))^2 )(\partial_x f (x, t))^2 dx,
$$
under \eqref{modelodiego}, we find
$$
\|\pax f\|_{L^2}^2+\frac{1}{6}\|\pax f\|_{L^4}^4+\int_0^t\|\pax f(s)\|^2_{\dot{H}^{0.5}}ds=\|\pax f_0\|_{L^2}^2+\frac{1}{6}\|\pax f_0\|_{L^4}^4.
$$
We thank the anonymous referee for pointing out this energy. This energy balance can be used to extend Theorem \ref{teoweak} to arbitrary (non necessarily positive) $g\in L^2\cap L^4$.
\end{remark}

\subsection{Plan of the paper}

The structure of the paper is as follows: In section \ref{secmodelodiego}, we prove the energy balance \eqref{energybalance} for the solutions of equation \eqref{modelodiego} and we use it to prove global existence of weak solutions of \eqref{modelodiegoderiv}. The results concerning \eqref{model1} are contained from Section \ref{secmodel} to Section \ref{seclarge}. In Section \ref{secmodel} we obtain well-posedness in Sobolev spaces and in an analytical framework for equation \eqref{model1}. In Section \ref{secqual} we study the qualitative properties of the solutions and we get some maximum principles for different lower order norms. In this Section, using the same scheme as in \cite{cheng2012global}, we also prove a global existence and decay in $C^1$ for the mild solution corresponding to small initial data in $H^3$. In Section \ref{seclim} we obtain existence and decay in $H^1$ for the mild solution corresponding to initial data small in $H^2$. In Section \ref{secdec} we present some numerics comparing the solutions to equations \eqref{full} and \eqref{eq0.1}. We present these simulations for the sake of completeness and to bring into comparison with the simulations corresponding to equation \eqref{model1}. In Section \ref{secnum} we present some numerics for equation \eqref{model1}. In particular we compute the evolution of a family of initial data reaching the boundary. In the last Section we study analytically some properties of the solutions when the initial data reaches the boundary. Notice that these solutions exist due to the well-posedness result in the analytical framework.

\section{A new energy balance and global weak solutions for \eqref{modelodiego}}\label{secmodelodiego}
Now we show a new energy balance for the derivative of \eqref{modelodiego}. 
\begin{proof}[Proof of Proposition \ref{prop3}]
We consider the equation for the derivative of the interface evolving in the infinite depth regime \eqref{modelodiegoderiv}. 

Now consider the evolution of the following quantity,
\begin{align*}
\frac{d}{dt}\int_\Omega g\log(1+g^2) &= \int g_t\log(1+g^2) + \int_\Omega\frac{2g^2g_t}{1+g^2}\\
&= \int_\Omega g_t\log(1+g^2) - 2\int_\Omega\frac{g_t}{1+g^2}\\
&= \int_{\Omega} g_t\log(1+g^2) - 2\int_\Omega \partial_t(\arctan(g))
\end{align*}
since $\int_\Omega g_t = 0$.

Let's look at the first term in the right hand side,
\begin{align}
&\int_\Omega g_t\log(1+g^2) = \int_\Omega -\Lambda g\log(1+g^2) + \pax \left(\frac{g^2Hg}{1+g^2}\right)\log(1+g^2)\label{E:aux2}
\end{align}
The second term,
\begin{align*}
&\int_\Omega\pax \left(\frac{g^2Hg}{1+g^2}\right)\log(1+g^2) = \int_\Omega \pax(Hg - \frac{Hg}{1+g^2})\log(1+g^2)\\
& = \int_\Omega \Lambda g\log(1+g^2) - \int_\Omega \pax\left(\frac{Hg}{1+g^2}\right)\log(1+g^2)\\
&= \int_\Omega \Lambda g\log(1+g^2) + \int_\Omega  Hg \frac{2gg_x}{(1+g^2)^2}\\
&=\int_\Omega \Lambda g\log(1+g^2) + \int_\Omega   \frac{\Lambda g}{1+g^2}
\end{align*}
Putting this back together into \eqref{E:aux2},
\begin{equation*}
\frac{d}{dt}\int_\Omega g\log(1+g^2)+2\frac{d}{dt}\int_\Omega \arctan(g) = \int_\Omega   \frac{\Lambda g}{1+g^2}.
\end{equation*}
\end{proof}

\begin{remark}
Now we can symmetrize the \emph{extra} term in \eqref{energybalance},
\begin{equation*}
\int_\TT\frac{\Lambda g}{1+g^2} = -\frac{1}{4\pi}\iint_{\TT\times\TT} \frac{(g(x)-g(y))^2(g(x)+g(y))}{\sin^2(\frac{x-y}{2})(1+(g(x))^2)(1+(g(y))^2)}dxdy,
\end{equation*}
in the periodic case and 
\begin{equation*}
\int_\RR\frac{\Lambda g}{1+g^2} = -\frac{1}{2\pi}\iint_{\RR\times\RR} \frac{(g(x)-g(y))^2(g(x)+g(y))}{(x-y)^2(1+(g(x))^2)(1+(g(y))^2)}dxdy,
\end{equation*}

which is negative if we assume that $g\geq 0$. This observation will allow us to gain half a derivative from this energy identity. 
\end{remark}

We fix $\Omega=\TT$ to simplify and we consider $f_0\in W^{1,\infty}$. Consequently $g_0\in L^\infty$. We also assume $g_0>0$. In particular this implies that 
$$
\|g(t)\|_{L^\infty(\TT)}\leq\|g_0\|_{L^\infty(\TT)},\;\min_x g(x,t)>\min_x g_0(x)\;\forall\,t>0.
$$

We use the previous energy identity to get compactness and to construct weak solutions:

\begin{proof}[Proof of Theorem \ref{teoweak}]
We define the approximate problems
\begin{equation}\label{E:smoothg}
g^\epsilon_t = -\Lambda g^\epsilon + \pax\left(\frac{(g^\epsilon)^2Hg^\epsilon}{1+(g^\epsilon)^2}\right) + \epsilon \pax^2g^\epsilon,\;g^\epsilon(x,0)=\rho_\epsilon*g_0,
\end{equation}
where $\rho_\epsilon$ is a standard mollifier.

Multiplying equation \eqref{E:smoothg} by $-\pax^2g^\epsilon$, and integrating over the torus, we obtain,

\begin{align*}
-\int_\TT g^\epsilon_t\pax^2g^\epsilon &= \int_\TT\Lambda g^\epsilon g^\epsilon_{xx} - \int_\TT \left(\frac{(g^\epsilon)^2Hg^\epsilon}{1+(g^\epsilon)^2}\right)_x g^\epsilon_{xx} - \epsilon\|g^\epsilon\|^2_{\dot{H}^2}\\
\frac{1}{2}\frac{d}{dt}\|g^\epsilon\|^2_{\dot{H}^1} &=\int_\TT\frac{\Lambda g^\epsilon}{1+(g^\epsilon)^2}g^\epsilon_{xx} - \int_\TT \frac{Hg^\epsilon 2g^\epsilon g^\epsilon_x}{(1+(g^\epsilon)^2)^2}g_{xx} -\epsilon\|g_{xx}\|^2_{\dot{H}^2}.
\end{align*}
To estimates the remaining terms, we will use the following inequalities which are a direct consequence of Gagliardo-Nirenberg,
\begin{equation}\label{GN2}
\|u\|_{\dot{W}^{1,4}}^2\leq C\|u\|_{\dot{H}^{2}}\|u\|_{L^\infty}\text{ and }\|u\|_{\dot{W}^{1,3}}^3\leq C\|u\|^2_{\dot{H}^{1.5}}\|u\|_{L^\infty}.
\end{equation}
These inequalities are valid for zero mean, periodic functions, but as the $L^1$ norm of our solution propagates with the evolution, we can adapt the argument straightforwardly. Using this into our estimate,
\begin{align*}
\frac{1}{2}\frac{d}{dt}\|g_x\|^2_0 + \epsilon\|g_{xx}\|^2_0&\leq \|g_{xx}\|_0\|\Lambda g\|_0 + 2\|g_{xx}\|_0\|g_x\|_{L^4}\|Hg\|_{L^4}\|g\|_{L^\infty}\\
&\leq C( \lambda \|g_{xx}\|^2_0 + \frac{1}{\lambda}\|g\|^{10}_{L^\infty}).
\end{align*}
Choosing $\lambda = \epsilon/2C$, we absorb the second derivative into the left side, and integrating in time we obtain,
\begin{equation}
\|g^\epsilon_x(t)\|^2_0 + \epsilon\|g^\epsilon_{xx}\|^2_{L^2_t L^2_x} \leq \|g^\epsilon_x(0)\|^2_0 + \frac{C}{\epsilon}t\|g^\epsilon\|^{10}_{L^\infty_t L^\infty_x}.
\end{equation}
Since the $L^\infty$-norm of $g^\epsilon$ is uniformly bounded, we have a global estimate for the $H^1$ norm of $g^\epsilon$ for every $\epsilon>0$. 

We study the evolution of $\int_\TT g^\epsilon\log(1+(g^\epsilon)^2)$. We find
\begin{multline*}
\int_\TT g^\epsilon(t)\log(1+(g^\epsilon(t))^2)dx+2\int_\TT \arctan(g^\epsilon)\\
-\int_0^t\int_\TT   \frac{\Lambda g^\epsilon}{1+(g^\epsilon)^2}dxds+2\epsilon\int_0^t\int_\TT\frac{g^\epsilon(\pax g^\epsilon)^2}{1+(g^\epsilon)^2}dxdt\\
=\int_\TT g^\epsilon_0\log(1+(g^\epsilon_0)^2)dx+2\int_\TT \arctan(g^\epsilon_0).
\end{multline*}
This implies the uniform-in-$\epsilon$ bound
$$
\int_0^t\|g^\epsilon(s)\|_{\dot{H}^{0.5}}^2\leq\frac{(1+\|g_0\|_{L^\infty})^2}{\min_x g_0}\left(\int_\TT g_0\log(1+g_0^2)dx+2\int_\TT \arctan(g_0)\right).
$$
Banach-Alaoglu Theorem implies $g\in L^\infty([0,T],L^\infty)\cap L^2([0,T],H^{0.5})$. Using \eqref{modelodiegoderiv} we get a uniform bound for $\pat g^\epsilon$ in $L^2_tH^{-1}_x$. Thus, we can apply Lemma \ref{lemma:2.1} with
$$
X_0=H^{0.5},\;X=L^2,\;Y=H^{-1},
$$
to get strong convergence $g^\epsilon\rightarrow g$ in $L^2([0,T]\times\TT)$. This compactness implies the convergence of the weak formulations. 
\end{proof}

\section{Well-posedness}\label{secmodel}
\subsection{Well-posedness in Sobolev spaces}
First, we prove local well-posedness in the stable regime:
\begin{proof}[Proof of Theorem \ref{teo1}]
We proof the case $s=3, l=\pi/2$ being the other cases analogous. We define the energy
\begin{equation}\label{eqenergy}
E(t)=\|f(t)\|_{H^3}+\|d[f]\|_{L^\infty}+\|D[f]\|_{L^\infty},
\end{equation}
where
\begin{equation}\label{eqdistance}
d[f]=\frac{1}{\left(\frac{\pi}{2}\right)^2-\left(f(x)\right)^2-(\pax f(x))^2},
\end{equation}
and
\begin{equation}\label{eqDistance}
D[f]=\frac{1}{\left(\frac{\pi}{2}\right)^2-\left(f(x)\right)^2}.
\end{equation}
The quantity $d[f]$ controls the stability condition \eqref{eqRayleigh} for our model. Indeed, if initially $d[f_0]>0$, then, as long as the energy remains bounded, $d[f]>0$. This implies that the dynamics is in the stable regime \eqref{eqRayleigh}. The quantity $D[f]$ ensures that we don't leave the set $H^s_l$.

\textbf{Estimates for $\pat f,\pat\pax f$:} By the basic properties of the Hilbert transform and the Sobolev embedding, we get
\begin{equation}\label{patf}
\|\pat f\|_{L^\infty}\leq 2\|H\pax f\|_{L^\infty}\leq C\|f\|_{H^{3}},
\end{equation}
and
\begin{multline}\label{patpaxf}
\|\pat \pax f\|_{L^\infty}\leq 2\|H\pax^2 f\|_{L^\infty}\\
+C\left(\|\pax f\|_{L^\infty}\|\pax^2 f\|_{L^\infty}+\|f\|_{L^\infty}\|\pax f\|_{L^\infty}\right)\|\Lambda f\|_{L^\infty}\\
\leq C\|f\|_{H^3}(\|f\|_{H^3}^2+1).
\end{multline}

\textbf{Estimates for $d[f]$:} We compute 
$$
\frac{d}{dt}d[f]=(d[f])^2\left(2f(x)\pat f(x)+2\pax f(x)\pat\pax f(x)\right).
$$
Using the definition of the energy \eqref{eqenergy}, we get
$$
\frac{d}{dt}d[f]\leq Cd[f]E\left(E+1\right)^4,
$$
thus, integrating in $(t,t+h)$
$$
d[f](t+h)\leq d[f](t)e^{C\int_t^{t+h} E\left(E+1\right)^4}ds.
$$
We have
\begin{multline}\label{patdf}
\frac{d}{dt}\|d[f]\|_{L^\infty}=\lim_{h\rightarrow0}\frac{\|d[f](t+h)\|_{L^\infty}-\|d[f](t)\|_{L^\infty}}{h}\\
\leq \|d[f](t)\|_{L^\infty}\lim_{h\rightarrow0}\frac{e^{C\int_t^{t+h} E\left(E+1\right)^4}ds-1}{h}\leq CE^2\left(E+1\right)^4.
\end{multline}

\textbf{Estimates for $D[f]$:}
In the same way,
\begin{equation}\label{patDf}
\frac{d}{dt}\|D[f]\|_{L^\infty}\leq CE^2\left(E+1\right)^4.
\end{equation}

\textbf{Estimates for the higher order terms:}
The higher order terms are
$$
I_1=\int_\Omega \left(\frac{1}{1+\left(\frac{\pi}{2}\right)^2-\left(f(x)\right)^2}-\frac{1}{1+\left(\pax f(x)\right)^2}\right)\Lambda \pax^3 f(x)\pax^3f(x)dx,
$$
$$
I_2=\int_\Omega \frac{2\pax f(x)\pax^4 f(x)}{\left(1+\left(\pax f(x)\right)^2\right)^2}\Lambda f(x)\pax^3f(x)dx.
$$

Notice that, due to \eqref{patdf}, $\sigma(t)<0$ for sufficiently small time. To estimate $I_1$ we use the pointwise inequality \cite{cordoba2003pointwise, cor2}
$$
2\theta\Lambda\theta\geq \Lambda\theta^2.
$$
This inequality and the self-adjointness of the operator allow us \emph{to integrate by parts} in the stable regime (which is guaranteed for a short time by \eqref{patdf}). We get
$$
I_1=J_1+J_2,
$$
with
\begin{eqnarray}\label{eqhigher}
J_1&\leq& \frac{1}{2}\int_\Omega \Lambda\left(\frac{1}{1+\left(\frac{\pi}{2}\right)^2-\left(f(x)\right)^2}-\frac{1}{1+\left(\pax f(x)\right)^2}-\sigma(t)\right)\nonumber\\
&&\times (\pax^3f(x))^2dx\nonumber\\
&\leq& C\|f\|^2_{H^3}\left\|\Lambda \left[\frac{1}{1+\left(\frac{\pi}{2}\right)^2-\left(f(x)\right)^2}-\frac{1}{1+\left(\pax f(x)\right)^2}\right]\right\|_{H^{0.6}}\nonumber\\
&\leq& \left(\left\|\frac{2f\pax f}{(1+\left(\frac{\pi}{2}\right)^2-\left(f(x)\right)^2)}\right\|_{H^{1}}+\left\|\frac{2\pax f \pax^2f}{(1+\left(\pax f(x)\right)^2)}\right\|_{H^{1}}\right)\nonumber\\
&&\times C\|f\|^2_{H^3},
\end{eqnarray}
and
\begin{equation}\label{eqhigher2}
J_2= \sigma(t)\int_\Omega (\Lambda^{0.5}\pax^3f(x))^2dx<0.
\end{equation}
The term $I_2$ can be bounded as in \cite{c-g-o08}. With the same ideas, we can handle the lower order terms. We conclude 
$$
\frac{d}{dt}\|f(t)\|_{H^3}\leq C(E+1)^7.
$$

\textbf{Obtaining uniform estimates:} Collecting the estimates (see \eqref{patdf}, \eqref{patDf}, \eqref{eqhigher}, \eqref{eqhigher2}), we get
$$
\frac{d}{dt}E\leq C(E+1)^7.
$$
Using Gronwall's inequality, we obtain 
$$
E(t)\leq C(f_0),\text{ if }0\leq t \leq T(f_0).
$$
With this a priori estimate we can obtain the local existence of smooth solutions using the standard arguments (see \cite{bertozzi-Majda}). Moreover, \eqref{eqhigher} and \eqref{eqhigher2} give us 
$$
\frac{d}{dt}\|f(t)\|^2_{H^3}-\sigma(t)\|f(t)\|^2_{\dot{H}^{3.5}}\leq C(E+1)^k.
$$
Integrating in time, we conclude $f\in L^2_tH^{3.5}_x$.

\textbf{Uniqueness:} Let's assume that there exists $f_1,$ $f_2$, two different solutions corresponding to the same initial data and denote $f=f_1-f_2$. Then, with the same ideas, we get
$$
\frac{d}{dt}\|f\|_{L^2}^2\leq C(\|f_1\|_{H^3},\|f_2\|_{H^3})\|f\|_{L^2}^2,
$$
and using Gronwall inequality, we conclude the uniqueness.

\textbf{Continuation criterion:} We use Lemma \ref{lemaaux} ($\alpha=1,d=1,\gamma=0.6$) in \eqref{eqhigher} to get
\begin{equation}\label{eqbound}
\Lambda\left[\frac{1}{1+l^2-(f(x))^2}\right]\leq C(l)\|f\|_{W^{0.6,\infty}}^2+\frac{2f(x)}{(1+l^2-(f(x))^2)^2}\Lambda f(x),
\end{equation}
\begin{equation}\label{eqbound2}
\Lambda\left[\frac{-1}{1+(\pax f(x))^2}\right]\leq C\|\pax f\|_{W^{0.6,\infty}}^2+\frac{2\pax f(x)}{(1+(\pax f(x))^2)^2}\Lambda \pax f(x).
\end{equation}
From here we conclude the result.

\end{proof}

\subsection{Well-posedness in the analytical framework}
We start with a useful Lemma:
\begin{lem}\label{CKprop}
Consider $0\leq r'<r$ and the set $\mathcal{O}^\tau_R$. Then, for $f,g\in \mathcal{O}^\tau_R$, the spatial operator in \eqref{eqmodelIA}, $F:\mathcal{O}^\tau_R\rightarrow H^3(\Ss_{r'})$ is continuous. Moreover, the following inequalities holds:
\begin{enumerate}
 \item $\|F[f]\|_{H^3(\Ss_{r'})}\leq\frac{C^\tau_R}{r-r'}\|f\|_{H^3(\Ss_r)},$
 \item $\|F[f]-F[g]\|_{H^3(\Ss_{r'})}\leq\frac{C^\tau_R}{r-r'}\|f-g\|_{H^3(\Ss_r)}.$
\end{enumerate}
\end{lem}
\begin{proof}
For the sake of brevity, we only prove the first part. The second one is analogous. Notice that
$$\|\Lambda f\|_{L^2(\Ss_r')}\leq\frac{C}{r-r'}\|Hf\|_{L^2(\Ss_r)}\leq\frac{C}{r-r'}\|f\|_{L^2(\Ss_r)}$$
By definition:
$$
\|F[f]\|^2_{H^3(\Ss_r')}=I+II
$$
where
$$
I=\left\|\left(\frac{-1}{1+(\partial_x f)^2}+\frac{1}{1+l^2-f^2} \right)\Lambda f\right\|_{L^2(\Ss_r')}^2
$$
$$
II=\left\|\pax^3\left(\left(\frac{-1}{1+(\partial_x f)^2}+\frac{1}{1+l^2-f^2} \right)\Lambda f\right)\right\|_{L^2(\Ss_r')}^2
$$
To estimate $I$, we use H\"older and the fact that we're working in the open set $\mathcal{O}^\tau_R$ to get
$$
1-\tau\leq|1+(\partial_x f)^2|^2,
$$
and, as a consequence,
$$
\left|\frac{1}{1+(\partial_x f)^2}\right|
\leq \frac{1}{\sqrt{1-\tau}}.
$$
A similar bound holds for the term $\frac{1}{1+l^2-f^2}$. Hence, we have:
$$
\left\|\left(\frac{-1}{1+(\partial_x f)^2}+\frac{1}{1+l^2-f^2} \right)\Lambda f\right\|_{L^2(\Ss_r')}\leq\frac{C^\tau_R}{r-r'}\|f\|_{L^2(\Ss_r)}
$$
In $II$ we compute the third derivative. The terms involving $1$,$2$ and $3$ derivatives can be bounded using the previous ideas, the open set definition and the Banach scale property. For the terms involving $4$ derivatives we use \eqref{cauchy2}. In particular
\begin{multline*}
\|(1+(\partial_x f)^2)^{-2}\partial_x f\partial_x^4 f\|_{L^2(\Ss_r')}\\
\leq\|(1+(\partial_x f)^2)^{-2}\partial_x f\|_{L^\infty(\Ss_r')}\|\partial_x(\partial_x^3 f)\|_{L^2(\Ss_r')}\leq \frac{C^\tau_R}{r-r'}\|f\|_{H^3(\Ss_r)},
\end{multline*}
\begin{multline*}
\left\|\left(\frac{-1}{1+(\partial_x f)^2}+\frac{1}{1+l^2-f^2} \right)\partial_x^3\Lambda f\right\|_{L^2(\Ss_r')}\\
\leq\left\|\left(\frac{-1}{1+(\partial_x f)^2}+\frac{1}{1+l^2-f^2} \right)\right\|_{L^\infty(\Ss_r')}\|\partial_x^3(\Lambda f)\|_{L^2(\Ss_r')}^2\\
\leq\frac{C^\tau_R}{r-r'}\|f\|_{H^3(\Ss_r)}.
\end{multline*}
This concludes the proof.
\end{proof}

The former Lemma is used in the proof of Theorem \ref{teoCK}

\begin{proof}[Proof of Theorem \ref{teoCK}]
The proof follows the ideas in \cite{nirenberg1972abstract, nishida1977note} (see also \cite{ccfgl, CGO, marchioro1994mathematical}. We fix $R>R_0>0$ and $1>\tau>\tau_0$ and we consider the following Picard's iteration scheme
$$
f^n(t)=f_0+\int_0^tF[f^{n-1}]ds.
$$
By induction hypothesis we have $f^{j}\in \mathcal{O}^\tau_R$ for $1<j<n$. Using Lemma \ref{CKprop} and the ideas in \cite{nirenberg1972abstract, nishida1977note} we can find $T_0>0$ such that $\|f^n\|_{H^3(\Ss_r)}<R$, consequently, we need to find $T_1,T_2$ such that
$$
0<d_1[f^n]<R,\;0<d_2[f^n]<R.
$$
We obtain $T_2$ for $d_2$, being $T_1$ similar. We have 
$$
\tau-\im\pax f^n\geq\tau-\im\pax f^0-tC^\tau_R=\tau-\tau_0+\tau_0-\im\pax f^0-tC^\tau_R\geq \frac{1}{R_0}-tC^\tau_R>\frac{1}{R},
$$
by taking $0\leq t\leq T_2=T_2(R,\tau)$ small enough. We define $T=\min\{T_0,T_1,T_2\}>0$ and we conclude.
\end{proof}

\section{Qualitative theory}\label{secqual}

\subsection{Decay estimates for the lower norms}

\begin{proof}[Proof of Proposition \ref{prop1}] We assume $l=\pi/2$ without losing generality.

\textbf{Step 1:} The proof of this part is straightforward.

\textbf{Step 2:} We denote $M(t)=\max_x f(x,t)=f(X_t,t)$ and $m(t)=\min_x f(x,t)=f(x_t,t)$. Using Rademacher Theorem as in \cite{cor2} and \cite{c-g09}, we obtain
$$
\frac{d}{dt}f(X_t)=-\frac{\left(\frac{\pi}{2}\right)^2-\left(f(X_t)\right)^2}{1+\left(\frac{\pi}{2}\right)^2-\left(f(X_t)\right)^2}\Lambda f(X_t).
$$ 
Using the kernel representations
$$
\Lambda u(x)=\frac{1}{\pi}\text{P.V.}\int_\RR\frac{u(x)-u(x-\eta)}{\eta^2}d\eta,\text{ }
\frac{1}{2\pi}\text{P.V.}\int_\TT\frac{u(x)-u(x-\eta)}{\sin^2\left(\frac{\eta}{2}\right)}d\eta,
$$
for the flat at infinity case and for the periodic case, respectively. We conclude $M'(t)<0$. With the same approach we get $m'(t)>0$.

\textbf{Step 3:} The solution remains odd, so, as in \cite{AGM}, we have
\begin{multline*}
\Lambda f(X_t)=\frac{1}{2\pi}\text{P.V.}\int_\TT\frac{f(X_t)-f(X_t-\eta)}{\sin^2(\eta/2)}d\eta\\
\geq\frac{1}{2\pi}\text{P.V.}\int_\TT f(X_t)-f(X_t-\eta)d\eta=f(X_t).
\end{multline*}
Thus, we have
$$
\frac{\left(\frac{\pi}{2}\right)^2-\left(f(X_t)\right)^2}{1+\left(\frac{\pi}{2}\right)^2-\left(f(X_t)\right)^2}\Lambda f(X_t)\geq \frac{\left(\frac{\pi}{2}\right)^2-\left(f(X_t)\right)^2}{1+\left(\frac{\pi}{2}\right)^2}f(X_t)\geq \frac{\left(\frac{\pi}{2}-f(X_t)\right)}{1+\left(\frac{\pi}{2}\right)^2}f(X_t).
$$
We conclude
$$
\frac{d}{dt}f(X_t)\leq -\frac{\left(\frac{\pi}{2}-f(X_t)\right)}{1+\left(\frac{\pi}{2}\right)^2}f(X_t),
$$ 
therefore
$$
\|f(t)\|_{L^\infty(\TT)}\leq \frac{\pi \|f_0\|_{L^\infty(\TT)}}{\pi e^{\frac{\pi}{2}\frac{t}{1+\left(\frac{\pi}{2}\right)^2}}+2\|f_0\|_{L^\infty(\TT)}\left(1-e^{\frac{\pi}{2}\frac{t}{1+\left(\frac{\pi}{2}\right)^2}}\right)}.
$$ 

\textbf{Step 4:} We test the equation \eqref{model1} against $\Lambda f$, integrate in space and use the self-adjointness. Recalling the definition of $\sigma(t)$ \eqref{sigma}, we have
$$
\|f(t)\|^2_{\dot{H}^{0.5}}-2\int_0^t \sigma(s)\|f(s)\|^2_{\dot{H}^{1}}\leq \|f_0\|^2_{\dot{H}^{0.5}}.
$$
Assume again that the solution doesn't leave the stable regime up to time $T$, then testing equation \eqref{model1} against $f$, we get
$$
\frac{d}{dt}\|f\|_{L^2}^2\leq 4\|\Lambda f\|_{L^2}\|f\|_{L^2}, 
$$
and integrating in time,
\begin{multline*}
\|f(t)\|_{L^2}\leq \|f_0\|_{L^2}+ \frac{4}{\min_{0\leq s\leq T}\sqrt{|\sigma(s)|}}\int_0^t\sqrt{-\sigma(s)}\|f(s)\|_{\dot{H}^1}ds\\
\leq \|f_0\|_{L^2}+ \frac{2\sqrt{T}\|f_0\|_{\dot{H}^{0.5}}}{\min_{0\leq s\leq T}\sqrt{|\sigma(s)|}}.
\end{multline*}

\textbf{Step 5:} Using Poincar\'e inequality and recalling that fractional derivatives have zero mean, we get
$$
\frac{1}{2}\frac{d}{dt}\|f(t)\|^2_{\dot{H}^{0.5}(\TT)}\leq \sigma(t)\|f(t)\|^2_{\dot{H}^{0.5}(\TT)}.
$$
Using Gronwall inequality we conclude the result.

\textbf{Step 6:} Taking $T$ such that the solution is in the stable regime and using a previous step, we have
$$
\|f(t)\|_{L^2}\leq \mathfrak{C}(f_0,T).
$$
Using Rademacher Theorem and the decay of the amplitude, we obtain  
\begin{eqnarray*}
\frac{d}{dt}\|f(t)\|_{L^\infty(\RR)}&=&-\left(\frac{\left(\frac{\pi}{2}\right)^2-\|f(t)\|_{L^\infty(\RR)}^2}{1+\left(\frac{\pi}{2}\right)^2-\|f(t)\|_{L^\infty(\RR)}^2}\right)\Lambda f(x_t)\\
&\leq& -\left(\frac{\frac{\pi}{2}-\|f(t)\|_{L^\infty(\RR)}}{1+\left(\frac{\pi}{2}\right)^2}\right)\frac{\|f(t)\|_{L^\infty(\RR)}^{3}}{\mathfrak{C}(f_0,T)}\\
&\leq& -\left(\frac{\frac{\frac{\pi}{2}}{\|f(t)\|_{L^\infty(\RR)}}-1}{1+\left(\frac{\pi}{2}\right)^2}\right)\frac{\|f(t)\|_{L^\infty(\RR)}^{4}}{\mathfrak{C}(f_0,T)}\\
&\leq& -\left(\frac{\frac{1-\mathcal{A}}{\mathcal{A}}}{1+\left(\frac{\pi}{2}\right)^2}\right)\frac{\|f(t)\|_{L^\infty(\RR)}^{4}}{\mathfrak{C}(f_0,T)},
\end{eqnarray*}
thus, for $0\leq t\leq T$ we conclude
$$
\|f(t)\|_{L^\infty(\RR)}\leq \frac{1}{\sqrt[3]{\left(\frac{\frac{1-\mathcal{A}}{\mathcal{A}}}{1+\left(\frac{\pi}{2}\right)^2}\right)\frac{3t}{\mathfrak{C}(f_0,T)}+\frac{1}{\|f_0\|^3_{L^\infty(\RR)}}}}
$$
\end{proof}

\subsection{Global existence and decay estimates in $C^1$}
\begin{proof}[Proof of Theorem \ref{teo3}]
In this case we have $X=H^3$, $Y=C^1$ and $\mathcal{D}(t)=e^{\mathcal{C}_lt}$. We use the estimate
\begin{equation}\label{normlinearoperator}
\|e^{-t\mathcal{C}_l\Lambda}\|_{L^\infty\rightarrow L^\infty}\leq e^{-\mathcal{C}_lt},
\end{equation}
and get
\begin{equation}\label{boundduhamel}
\|f(t)\|_{Y}\leq\|f_0\|_{Y}e^{-\mathcal{C}_lt}+\int_0^te^{-\mathcal{C}_l(t-s)}\|NL(s)\|_{Y}ds.
\end{equation}
Using Sobolev embedding, we have
\begin{equation*}%\label{boundNL}
\|NL(s)\|_{L^\infty}\leq C(l)\|f\|_{H^{1.6}}\left(\|\pax f\|^2_{L^\infty}+\|f(t)\|_{L^\infty}^2\right)\leq
C(l)\frac{\vertiii{f}_T^3}{(\mathcal{D}(t))^{2}}.
\end{equation*}
and
\begin{multline*}
\|\pax NL(s)\|_{L^\infty}\leq C(l)\left[\|f\|_{H^{2.6}}\left(\|\pax f\|^2_{L^\infty}+\|f(t)\|_{L^\infty}^2\right)\right.\\
+\|f\|_{H^{1.6}}\left(\|\pax f\|_{L^\infty}\|\pax^2 f\|_{L^\infty}\left(\|\pax f\|^2_{L^\infty}+\|f(t)\|_{L^\infty}^2\right)\right.\\
+\|\pax f\|_{L^\infty}\|f\|_{L^\infty}\left(\|\pax f\|^2_{L^\infty}+\|f(t)\|_{L^\infty}^2\right)\\
\left.\left.+\|\pax f\|_{L^\infty}\|\pax^2 f\|_{L^\infty}+\|\pax f\|_{L^\infty}\|f\|_{L^\infty}\right)\right].
\end{multline*}
Recalling the following inequalities
\begin{equation}\label{GN}
\|u\|_{H^{s}}\leq C\|u\|^{1-s/r}_{L^2}\|u\|^{s/r}_{H^r}\text{ and }\|u\|_{L^\infty}\leq C\|\pax u\|^{0.5}_{L^2}\|u\|^{0.5}_{L^2},
\end{equation}
and using $\|\pax f\|_{L^2}\leq C\|\pax f\|_{L^\infty}$, we get
$$
\|f\|_{H^{1.6}}\leq C\left(\|f\|_{L^\infty}+\|\pax f\|^{0.7}_{L^2}\|\pax f\|^{0.3}_{H^2}\right)\leq C\vertiii{f}_Te^{-\mathcal{C}_lt0.7},
$$
$$
\|\pax^2 f\|_{L^\infty}\leq C\|\pax f\|^{0.5}_{L^2}\|\pax^3 f\|^{0.5}_{L^2}\leq C\vertiii{f}_Te^{-\mathcal{C}_lt0.5}.
$$
Putting all the estimates together, we conclude the following estimate
\begin{equation*}%\label{NLbound}
\|NL(s)\|_{C^1}\leq C(l)\left(\vertiii{f}_T^3+\vertiii{f}_T^5\right)e^{-2\mathcal{C}_ls}.
\end{equation*}
Inserting this estimate in \eqref{boundduhamel}, we obtain
$$
e^{\mathcal{C}_lt}\|f(t)\|_{C^1}\leq\|f_0\|_{C^1}+C(l)\left(\vertiii{f}_T^3+\vertiii{f}_T^5\right)\int_0^\infty e^{-\mathcal{C}_ls}ds.
$$
With the energy estimates in Theorem \ref{teo1} and the definition of \eqref{totalnorm}, we get
$$
\frac{d}{dt}\|f(t)\|_{H^3}^2\leq C\|f(t)\|_{H^3}\mathcal{P}(\vertiii{f}_T)e^{-\mathcal{C}_lt},
$$
for a polynomial $\mathcal{P}$ with powers bigger than one. Integrating in time and collecting all the estimates together, we conclude
$$
\vertiii{f}_T\leq C\|f_0\|_{H^3}+\|f_0\|_{C^1}+\mathcal{Q}(\vertiii{f}_T)\leq C\delta+\mathcal{Q}(\vertiii{f}_T),
$$
where $\mathcal{Q}$ is a polynomial with high powers. From this latter inequality,  by a standard continuation argument, we obtain the global existence if $\delta$ is small enough. Moreover, if we take $\delta$ small enough, we can ensure that
$$
e^{\mathcal{C}_lt}\|f(t)\|_{C^1}\leq \frac{l}{4}.
$$
\end{proof}

\section{Limited regularity}\label{seclim}
\begin{proof}[Proof of Theorem \ref{teo4}]
We explain how to obtain the good bounds, then, using mollifiers for the initial data the result follows. 

\textbf{Step 1: Case $\Omega=\TT$} Given $\delta>0$, we define the energy
$$
E(t)=\|f(t)\|_{H^2}+\|f(t)\|_{\dot{W}^{1,\infty}}+\frac{1}{2\delta-\|f(t)\|_{H^2}}.
$$
We define $M(t)=\max_x \pax f(x,t)=\pax f(x_t)$ as in Proposition \ref{prop1}. Then, using Rademacher's Theorem, we have
$$
M'=\left(\frac{1}{1+l^2-(f(x_t))^2}-\frac{1}{1+(\pax f(x_t))^2}\right)\Lambda\pax f(x_t)+\Lambda f(x_t)\frac{2f(x_t)\pax f(x_t)}{(1+l^2-(f(x_t))^2)^2},
$$
thus, using Proposition \ref{prop1},
$$
M'\leq \left(\frac{1+l^2-\|f_0\|_{L^\infty}+2\|f(t)\|_{H^2}\|f_0\|_{L^\infty}}{(1+l^2-\|f_0\|_{L^\infty}^2)^2}-\frac{1}{1+(\pax f(x_t))^2}\right)M(t).
$$
If $\|f_0\|_{H^2}\leq \delta,$ due to the form of the energy, there exist a time $T^*$ such that $\max_{0\leq t\leq T^*}\|f(t)\|_{H^2}\leq 2\delta$. At this step in the proof, this time may depend on the regularization parameter, but we are going to bound it uniformly. Consequently,
$$
M'\leq \left(\frac{1}{1+l^2-C\delta^2}+\frac{C\delta^2}{(1+l^2-\|f_0\|_{L^\infty})^2}-\frac{1}{1+C\delta^2}\right)M(t),
$$
and, if $\delta<<1$ we obtain $M(t)\leq M(0),\text{ for }0\leq t\leq T^*.$ In the same way we obtain reverse inequality for $m(t)=\min_x \pax f(x,t)=\pax f(x_t)$. So, 
$$
\|\pax f(t)\|_{L^\infty}\leq\|\pax f_0\|_{L^\infty}\leq C\delta\text{ for }0\leq t\leq T^*.
$$
From this decay we obtain that the solution relies in the stable regime. Recalling the definition \eqref{sigma}, we compute
\begin{eqnarray*}
\frac{1}{2}\frac{d}{dt}\|f\|^2_{\dot{H}^2}&\leq& \left((4\|f_0\|_{L^\infty}+8)\|f(t)\|_{\dot{W}^{1,\infty}}+\|f_0\|_{L^\infty}\right)\|f\|^2_{\dot{H}^2}\\
&&+(3+\|f(t)\|_{H^2}2(1+2\|f(t)\|_{\dot{W}^{1,\infty}}))\|f\|^3_{\dot{W}^{2,3}}\\
&&+\left(8(\|f_0\|^2_{L^\infty}+\|f\|_{\dot{W}^{1,\infty}})\|f(t)\|_{\dot{W}^{1,4}}^2\right)\|f\|_{\dot{H}^2}\\
&&+\frac{1}{2}\int_\Omega\Lambda\left[\frac{1}{1+l^2-(f(x))^2}-\frac{1}{1+(\pax f(x))^2}\right](\pax^2f(x))^2dx\\
&&+\sigma(t)\|f(t)\|^2_{\dot{H}^{2.5}}.
\end{eqnarray*}
We use interpolation \eqref{GN2} to obtain
\begin{eqnarray*}
\frac{1}{2}\frac{d}{dt}\|f\|^2_{\dot{H}^2}&\leq& C\delta\|f\|^2_{\dot{H}^2}+\sigma(t)\|f(t)\|^2_{\dot{H}^{2.5}}+(1+\delta)C\|f\|^2_{\dot{H}^{2.5}}\|f\|_{\dot{W}^{1,\infty}}\\
&&+\frac{1}{2}\int_\Omega\Lambda\left[\frac{1}{1+l^2-(f(x))^2}-\frac{1}{1+(\pax f(x))^2}\right](\pax^2f(x))^2dx.
\end{eqnarray*}
We use \eqref{eqbound} and \eqref{eqbound2}. Thus, using the $L^p$-boundedness of the singular integral operators,
\begin{eqnarray*}
\frac{1}{2}\frac{d}{dt}\|f\|^2_{\dot{H}^2}&\leq& C\delta\|f\|^2_{\dot{H}^2}+\sigma(t)\|f(t)\|^2_{\dot{H}^{2.5}}+(1+\delta)C\|f\|^2_{\dot{H}^{2.5}}\|f\|_{\dot{W}^{1,\infty}}\\
&& +C(l)\|f\|_{W^{0.6,\infty}}^2\|f\|_{\dot{H}^2}^2+\|f_0\|_{L^\infty}\|\Lambda f\|_{L^\infty}\|f\|_{\dot{H}^2}^2\\
&&+C\|\pax f\|_{W^{0.6,\infty}}^2\|f\|_{\dot{H}^2}^2+\|f\|_{\dot{W}^{1,\infty}}\|f\|_{\dot{W}^{2,3}}^3\\
&\leq& C\delta\|f\|^2_{\dot{H}^2}+\sigma(t)\|f(t)\|^2_{\dot{H}^{2.5}}+C\|f\|^2_{\dot{H}^{2.5}}\|f\|_{\dot{W}^{1,\infty}}\\
&&+C\|f\|_{H^{2.1}}^2\|f\|_{\dot{H}^2}^2\\
&\leq& C\delta\|f\|^2_{\dot{H}^2}+\sigma(t)\|f(t)\|^2_{\dot{H}^{2.5}}+C\|f\|^2_{\dot{H}^{2.5}}\|f\|_{\dot{W}^{1,\infty}}\\
&&+C\|f\|^{0.32}_{L^2}\|f\|_{\dot{H}^2}^2\|f\|^{1.68}_{H^{2.5}}\\
&\leq& C\delta\|f\|^2_{\dot{H}^2}+\sigma(t)\|f(t)\|^2_{\dot{H}^{2.5}}+C\|f\|^2_{\dot{H}^{2.5}}\|f\|_{\dot{W}^{1,\infty}}\\
&&+C(\epsilon)\delta^2\|f\|_{\dot{H}^2}^{12.5}+C\epsilon\|f\|^{2}_{H^{2.5}}
\end{eqnarray*}
where we have used the continuous embedding $H^{2.1}\subset W^{1.6,\infty}$, \eqref{GN} and Young's inequality with $p=2/1.68$ and $q=6.25$. Notice that, if $\delta$ is small enough,
$$
\sigma(t)\leq \frac{1}{1+l^2-C\delta^2}-\frac{1}{1+C\delta^2}<0.
$$
Let $K\in\NN$ be a fixed number. Inserting the latter bound we get
\begin{multline*}
\frac{d}{dt}\|f\|^2_{\dot{H}^2}+\frac{1}{K(1+C\delta^2)}\|f\|^2_{\dot{H}^{2.5}}\leq C(l,\epsilon)\delta\|f\|^2_{\dot{H}^2}\\
+\|f\|^2_{\dot{H}^{2.5}}\left(\frac{1}{1+l^2-C\delta^2}-\frac{K-1}{K(1+C\delta^2)}+C\epsilon+C\delta\right).
\end{multline*}
Thus, taking $0<\delta,\epsilon<<1$ small enough and $K=K(l)$ large enough, we obtain 
$$
\|f\|_{\dot{H}^2}\leq\|f_0\|_{\dot{H}^2}e^{C(\epsilon,l)\delta t}.
$$
Putting all together, we obtain
$$
\frac{d}{dt}E\leq C(l)E(1+E^2)\leq C(l,\epsilon)\delta (1+E)^3.
$$
This bound doesn't depends on the regularization parameter, so using Gronwall inequality, we  obtain a time $T^*=T^*(l,\epsilon,\delta)$ where the solution remain in a ball with radius $2\delta$ in $H^2$. Moreover, due to the evolution of $M(t),m(t)$ and the Proposition \ref{prop1}, the solution doesn't leave the stable regime. This concludes the result in the periodic case.

\textbf{Step 2: Case $\Omega=\RR$} Given $u(x)$, we define $X$ such that $u(X)=\max_x u(x)$. Then
\begin{multline*}
\Lambda u(X)=\frac{1}{\pi}\text{P.V.}\int_\RR\frac{u(X)-u(X-y)}{y^2}dy\geq \frac{2u(X)}{\pi}\int_1^\infty y^{-2}dy\\
-\frac{1}{\pi}\int_{B^c(0,1)}\frac{u(X-y)}{y^2}dy\geq \frac{2}{\pi}u(X)-\frac{2}{3\pi}\|u\|_{L^2}
\end{multline*}
We define $M(t)$ as before. Then, using Rademacher's Theorem, we have
\begin{eqnarray*}
M'&\leq& \left(\frac{2\|f(t)\|_{H^2}\|f_0\|_{L^\infty}}{(1+l^2-\|f_0\|_{L^\infty}^2)^2}\right.\\
&&\left.+\frac{2}{\pi(1+l^2-\|f_0\|^2_{L^\infty})}-\frac{2}{\pi(1+(\pax f(x_t))^2)}\right)M(t)\\
&&-\left(\frac{1}{\pi(1+l^2-\|f_0\|^2_{L^\infty})}-\frac{1}{\pi(1+(\pax f(x_t))^2)}\right)\frac{2}{3}\|f\|_{\dot{H}^1}\\
&\leq& \left(\frac{C\delta^2}{(1+l^2-C\delta^2)^2}+\frac{2}{\pi(1+l^2-C\delta^2)}-\frac{2}{\pi(1+C\delta^2)}\right)M(t)\\
&&\left(\frac{-1}{\pi(1+l^2-C\delta^2)}+\frac{1}{\pi(1+C\delta^2)}\right)\frac{2}{3}\|f\|_{\dot{H}^1}
\end{eqnarray*}
and, if $\delta$ is taken small enough, integrating in time and using Proposition \ref{prop1}, we have
$$
\left(\frac{-1}{1+l^2-C\delta^2}+\frac{1}{1+C\delta^2}\right)\|f\|_{L^2_t\dot{H}^1}^2\leq-\int_0^t\sigma(s)\|f(s)\|_{\dot{H}^1}^2ds\leq \|f_0\|_{\dot{H}^{0.5}}^2,
$$
thus,
$$
\int_0^t\|f(s)\|_{\dot{H}^1}ds\leq \sqrt{t}\|f\|_{L^2_t\dot{H}^1}\leq \sqrt{t}C(l,\delta)\|f_0\|_{\dot{H}^{0.5}},
$$
and
$$
M(t)\leq M(0)+\sqrt{t}C(l,\delta)\|f_0\|_{\dot{H}^{0.5}}.
$$
With the same ideas we obtain the appropriate bound for $m(t)$ and we get
$$
\|f(t)\|_{\dot{W}^{1,\infty}}\leq \|f_0\|_{\dot{W}^{1,\infty}}+\sqrt{t}C(l,\delta)\|f_0\|_{\dot{H}^{0.5}}. 
$$
With this bound and Proposition \ref{prop1} we conclude the existence of a time $T_1=T_1(\delta,l)$ such that the solution doesn't leave the stable regime. We define the energy
$$
E(t)=\|f(t)\|_{H^2}+\frac{1}{2\delta-\|f(t)\|_{H^2}}.
$$
With the same ideas as in the periodic case we get a second time $T_2=T_2(\delta,l,\epsilon)$ such that the solution remains in the ball with center the origin and radius $2\delta$ in $H^2$. We take $T^*=\min(T_1,T_2)$ the time of existence and we conclude the result. 
\end{proof}
If we add a symmetry hypothesis for the initial data we can improve Theorem \ref{teo4}:
\begin{proof}[Proof of Theorem \ref{teo5}]
We have $X=H^{2.5}$ with norm $\|f(t)\|_X=\|f(t)\|_{H^2}+\int_0^t\|f(s)\|^2_{H^{2.5}}ds,$
$Y=H^1$ and $\mathcal{D}(t)=e^{\mathcal{C}_lt}$. Since Theorem \ref{teo4}, there exist a local solution on the interval $[0,T]$. We define the total norm $\vertiii{f}_T$ as in \eqref{totalnorm}. We use 
$$
\|e^{-\mathcal{C}_lt\Lambda}\|_{L^2\rightarrow L^2}\leq e^{-\mathcal{C}_lt},
$$
Due to the interpolation inequality \eqref{GN} and using the expression \eqref{nonlinear}, we get 
$$
\|NL(s)\|_{L^2}\leq C(l)\|\pax f\|_{L^2}\left(\|\pax f\|_{L^\infty}^2+\|f\|_{L^\infty}^2\right)\leq C(l)\frac{\vertiii{f}_T^3}{(\mathcal{D}(s))^{2}}.
$$
and, using Bochner Theorem,
$$
e^{\mathcal{C}_lt}\|f(t)\|_{L^2}\leq\|f_0\|_{L^2}+C(l)\vertiii{f}_T^3\int_0^t\frac{1}{e^{\mathcal{C}_ls}}ds\leq \|f_0\|_{L^2}+C(l)\vertiii{f}_T^3.
$$
With the same ideas, we get
$$
\|\pax NL(s)\|_{L^2}\leq C(l)\left(\frac{\vertiii{f}_T^{7/3}\|f(t)\|^{2/3}_{\dot{H}^{2.5}}}{(\mathcal{D}(t))^{4/3}}+\frac{\vertiii{f}_T^{3}+\vertiii{f}_T^{5}}{(\mathcal{D}(t))^{2}}\right).
$$
Consequently,
\begin{eqnarray*}
e^{\mathcal{C}_lt}\|f(t)\|_{\dot{H}^1}&\leq&\|f_0\|_{\dot{H}^1}+C(l)\left(\vertiii{f}_T^{3}+\vertiii{f}_T^{5}\right)\int_0^t\frac{1}{e^{\mathcal{C}_ls}}ds\\
&&+C(l)\vertiii{f}_T^{7/3}\int_0^t\frac{\|f(s)\|^{2/3}_{\dot{H}^{2.5}}ds}{e^{\mathcal{C}_ls(1/3)}}\\
&\leq& \|f_0\|_{\dot{H}^1}+C(l)\left(\vertiii{f}_T^{3}+\vertiii{f}_T^{5}\right)\\
&&+C(l)\vertiii{f}_T^{7/3}\left(\int_0^t\frac{1}{e^{\mathcal{C}_ls/2}}ds\right)^{2/3}\left(\int_0^t\|f(s)\|^{2}_{\dot{H}^{2.5}}ds\right)^{1/3}\\
&\leq&\|f_0\|_{\dot{H}^1}+C(l)\left(\vertiii{f}_T^{3}+\vertiii{f}_T^{5}+\vertiii{f}_T^{8/3}\right)
\end{eqnarray*}

We need to obtain \emph{bona fide} a priori estimates on the $H^2$ seminorm for small initial data. With the same estimates as in Theorem \ref{teo4} we get
$$
\frac{d}{dt}\|f\|^2_{\dot{H}^2}+\|f\|^2_{\dot{H}^{2.5}}\leq C(l,\|f_0\|_{H^2})\mathcal{P}(\vertiii{f}_T)e^{-\mathcal{C}_ls/4},
$$
where $C(l,\|f_0\|_{H^2})\rightarrow 1$ as $\|f_0\|_{H^2}\rightarrow0$ and $\mathcal{P}$ is a polynomial with high powers. Thus, adding both estimates,
$$
\vertiii{f}_T\leq \|f_0\|_{H^2}+ C(l,\|f_0\|_{H^2})\mathcal{Q}(\vertiii{f}_T).
$$
If $\|f_0\|_{H^2}<<1$ is choosen small enough, this nonlinear Gronwall-type inequality and the fact $\|f(t)\|_{W^{1,\infty}}\leq \|f_0\|_{W^{1,\infty}}$ (again, for a small enough $H^2$ initial data) give us the global existence by means of a classical continuation argument.
\end{proof}

\section{Numerical simulations for the confined Muskat problem}\label{secdec}
In this section we perform numerical simulations for equations \eqref{full} and \eqref{eq0.1} to study the decay of $\|f(t)\|_{L^ \infty}$. The main purpose of these simulations is to compare the behaviour when the depth is finite (equation \eqref{eq0.1}) with the case where the depth is infinite (equation \eqref{full}). We consider equations \eqref{eq0.1} and \eqref{full} where $\rho^2-\rho^1=4\pi$. For each initial datum we approximate the solutions of \eqref{eq0.1} and \eqref{full} with the same numerical and physical parameters. 

To perform the simulations we follow the ideas in \cite{c-g-o08}. The interface is approximated using cubic splines with $N$ spatial nodes. The spatial operator is approximated with Lobatto quadrature (using the function \emph{quadl} in Matlab). Then, three different integrals appear for a fixed node $x_i$: the integral between $x_{i-1}$ and $x_i$, the integral between $x_i$ and $x_{i+1}$ and the nonsingular ones. In the two first integrals we use Taylor series to remove the singularity. In the nonsingular integrals the integrand is made explicit using the splines. We use a classical explicit Runge-Kutta method of order 4 to integrate in time. In the simulations we take $N=300$ and $dt=10^{-3}$. In what follows we change slightly the notation and write $f^{\pi/2}(x,t)$ for the solution of \eqref{eq0.1} and $f^{\infty}(x,t)$ for the solution of \eqref{full}. Notice that the superscript denotes the depth in each situation. Then, given an initial datum $f(x,0)=f_0(x)$, which is the same for both evolution problems, we are computing a numerical approximation for $f^{\pi/2}(x,t)$ and $f^ \infty(x,t)$. The initial datum considered is
\begin{equation}\label{caso1confmus}
f_0(x)=\left(\frac{\pi}{2}-0.0001\right)e^{-x^6}.
\end{equation}
We obtain Figures \ref{deccase1}. We can see that the decay is slower in the finite depth case and the existence of a \emph{big} time interval with a very small decay.

\begin{figure}[h!]
		\begin{center}
		\includegraphics[scale=0.35]{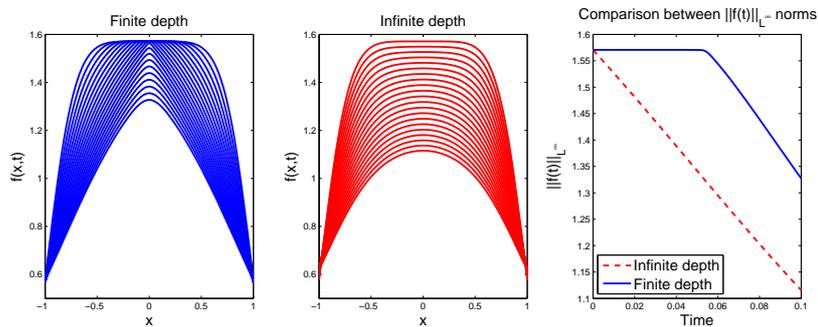} 
		\end{center}
		\caption[Numerics for the confined Muskat problem]{a) Dynamics for $f^{\pi/2}(x,t)$, b) dynamics for $f^\infty(x,t)$, c) $f^{\pi/2}(x,t)$ (blue) and $f^ \infty(x,t)$ (red) for the same times $t_i$ and initial datum given by \eqref{caso1confmus}.}
\label{deccase1}
\end{figure}

\section{Numerical simulations for \eqref{model1}}\label{secnum}
In this section we take $l=\pi/2$. To approximate the solutions to \eqref{model1} we use a Fourier collocation method with a explicit Runge-Kutta scheme for the time integration. We write $N$ for the number of spatial nodes. Then the operator $\Lambda$ can be easily discretized using the Fast Fourier Transform routine. To perform the multiplications we jump to the physical space. To advance in time we use a Runge-Kutta (4,5) scheme. 

\subsection{Decay of $\|f(t)\|_{L^\infty}$}
We consider $l=\pi/2$ and $f_0=(l-0.001)\cos(x)$ and we study $\|f(t)\|_{L^\infty}$. We show the results on Figure \ref{decaymodelfig}. We see that the evolution is qualitatively similar to the dynamics of the same quantity for equation \eqref{eq0.1} (see Figure \ref{deccase1}).
\begin{figure}[h!]
		\begin{center}
		\includegraphics[scale=0.2]{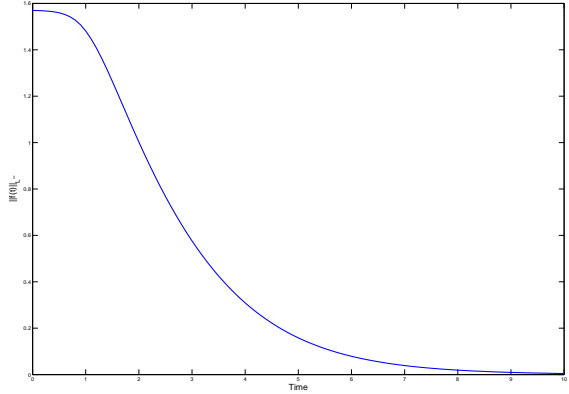} 
		\end{center}
		\caption{Evolution of $\|f(t)\|_{L^\infty}$}
\label{decaymodelfig}
\end{figure}

\subsection{Reaching the boundary} We consider $a>0$ and define the family of initial data
$$
f^a_0(x)=\cos(x)*a+\pi/2-a.
$$
Notice that $f^a_0(0)=\pi/2=l$, thus, the equation is in the unstable regime.

\begin{figure}[h!]
		\begin{center}
		\includegraphics[scale=0.25]{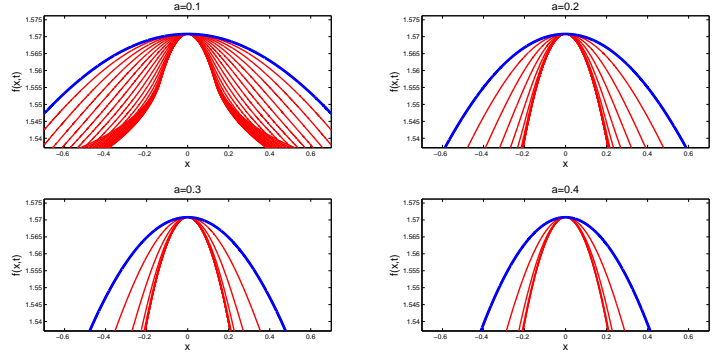} 
		\end{center}
		\caption{Evolution of $f(x,t)$ for the different cases $a=0.1,0.2,0.3,0.4$ with $N=2^{16}$. The wide line corresponds to the initial data.}
\label{evofig}
\end{figure}

\begin{figure}[h!]
		\begin{center}
		\includegraphics[scale=0.25]{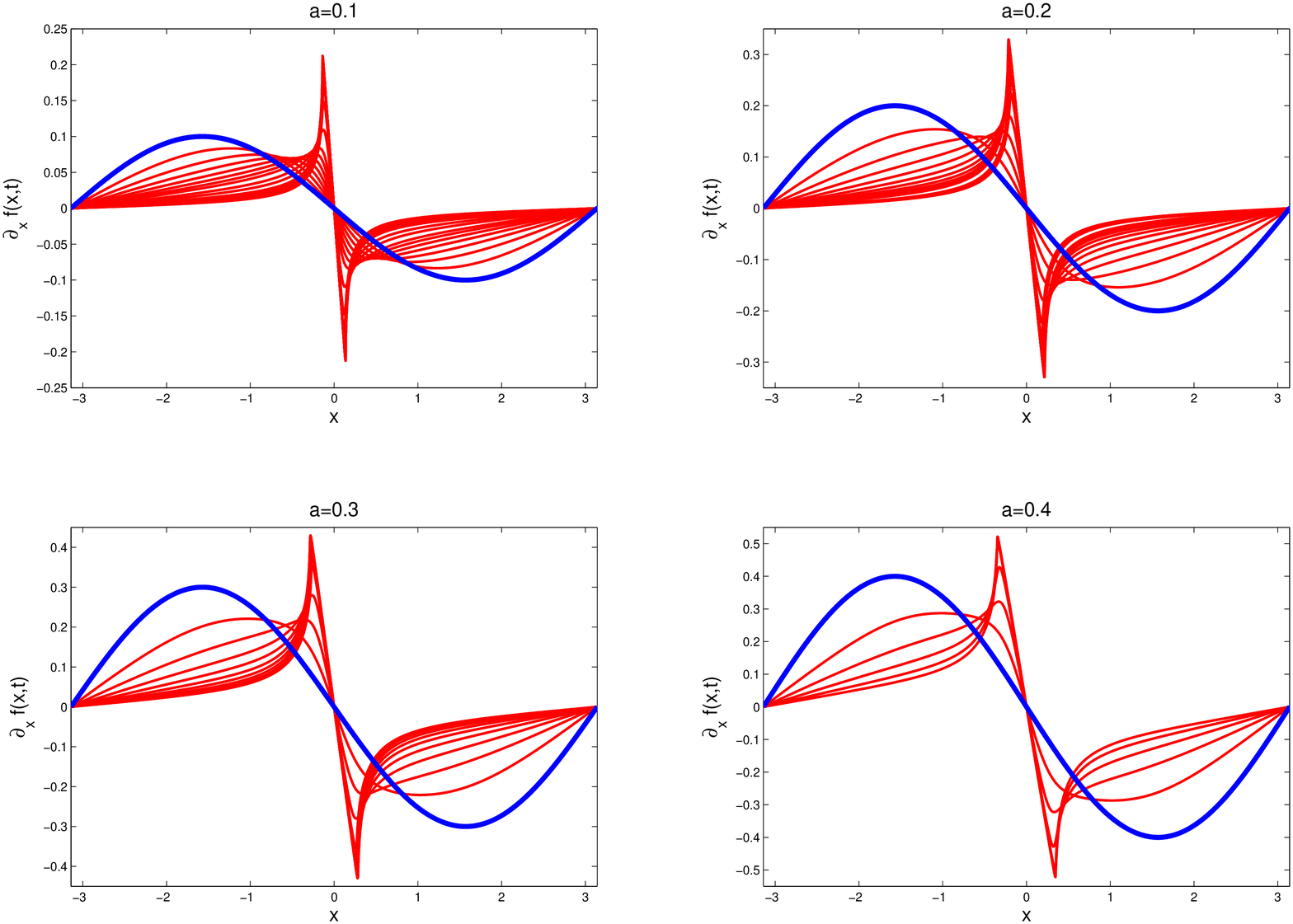} 
		\end{center}
		\caption{Evolution of $\pax f(x,t)$ for the different cases $a=0.1,0.2,0.3,0.4$ with $N=2^{17}$. The wide line corresponds to the initial data.}
\label{evodxfig}
\end{figure}

\begin{figure}[h!]
		\begin{center}
		\includegraphics[scale=0.25]{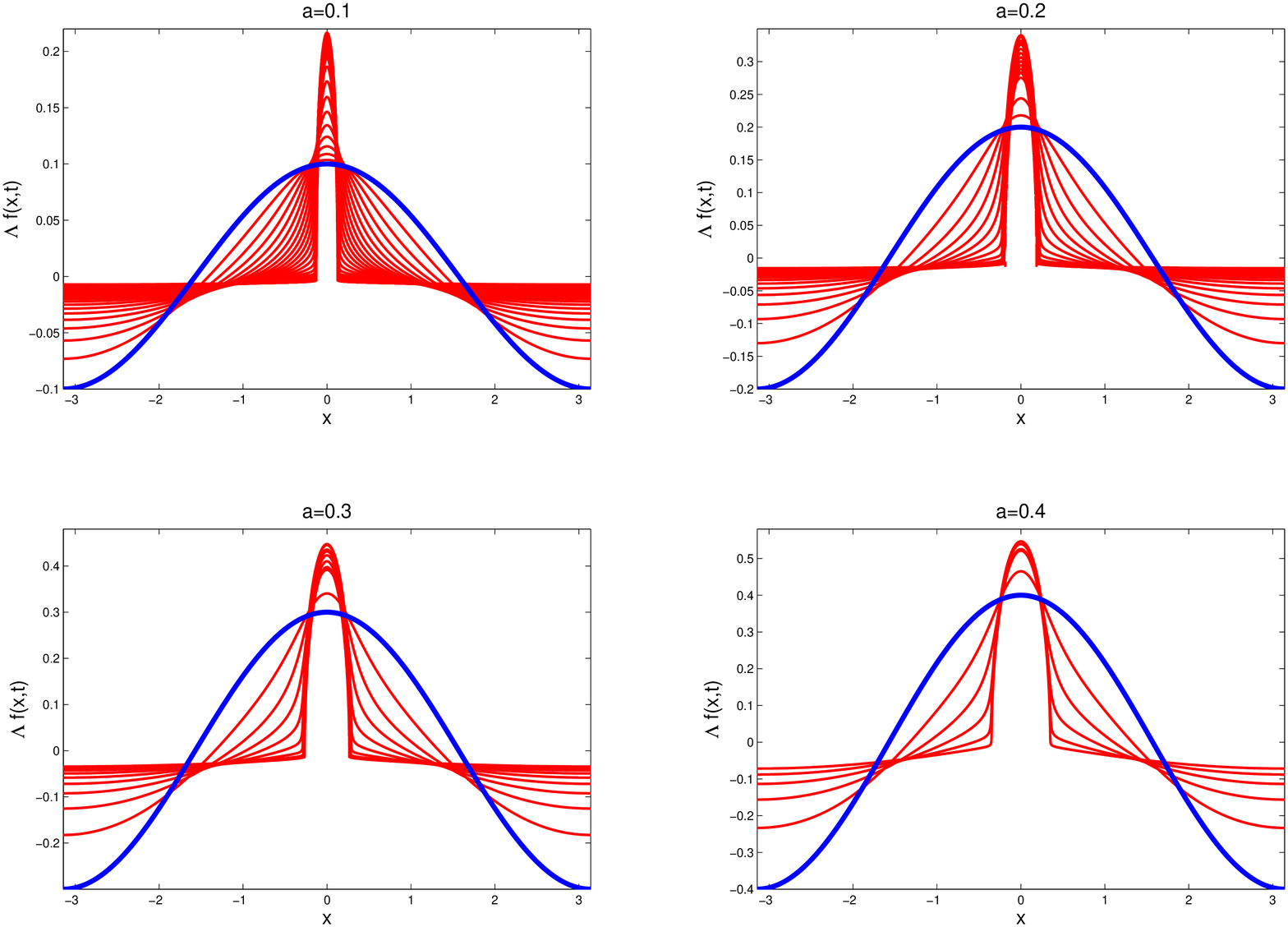} 
		\end{center}
		\caption{Evolution of $\Lambda f(x,t)$ for the different cases $a=0.1,0.2,0.3,0.4$ with $N=2^{17}$. The wide line corresponds to the initial data.}
\label{evolambdafig}
\end{figure}

We see in Figure \ref{evofig} that the second derivative at $x=0$ grows. In Figures \ref{evodxfig} and \ref{evolambdafig}, we observe that there exists two points where $\pax^2 f$ and $\pax\Lambda f$ are large. 

\section{Large time dynamics}\label{seclarge}
In this section we show that the solution never leaves the stable regime and that, for H\"older solutions, the curvature is bounded at the point where the initial data reach the boundary. These statement excludes the two main candidates for finite time singularities.

\begin{proof}[Proof of Proposition \ref{prop2}]
\textbf{Step 1:} We define
$$
\Sigma(x,t)=\frac{1}{1+l^2-(f(x,t))^2}-\frac{1}{1+(\pax f(x,t))^2}.
$$
We compute
\begin{equation}\label{eqsigma}
\pat \Sigma=\frac{-2f\Sigma\Lambda f}{(1+l^2-f^2)^2}+\frac{-2\pax f\left(\Sigma\Lambda\pax f+\pax \Sigma\Lambda f\right)}{(1+\pax f^2)^2}.
\end{equation}
Using \eqref{sigma} and Rademacher Theorem, we get
$$
\frac{d}{dt}\sigma(t)=-2\sigma(t)\left(\frac{f(x_t,t)\Lambda f(x_t,t)}{(1+l^2-(f(x_t,t))^2)^2}+\frac{\pax f(x_t,t)\Lambda\pax f (x_t,t)}{(1+(\pax f(x_t,t))^2)^2}\right),
$$
thus,
$$
\sigma(t)=\sigma(0)\exp\left(-2\int_0^t \left(\frac{f(x_s,s)\Lambda f(x_s,s)}{(1+l^2-(f(x_s,s))^2)^2}+\frac{\pax f(x_s,s)\Lambda\pax f (x_s,s)}{(1+(\pax f(x_s,s))^2)^2}\right)ds\right).
$$
From this equation we conclude the first statement.

\textbf{Step 2:} We consider an initial data such that $f(\tilde{x})=l$ for some $\tilde{x}$. Evaluating \eqref{eqsigma} at $\tilde{x}$ we get 
$$
\Sigma(\tilde{x},t)=\Sigma(\tilde{x},0)\exp\left(-2\int_0^t\frac{f(\tilde{x},s)\Lambda f(\tilde{x},s)}{(1+l^2-(f(x_s,s))^2)^2}ds\right).
$$
Assuming $f(x,t)\in C([0,T],H^3(\Omega))$ we obtain $\Sigma(\tilde{x},t)=0$ for all $0\leq t\leq T$. We compute
\begin{eqnarray*}
\pat\pax^2f&=&\Lambda\pax^2 f\left(\frac{1}{1+l^2-f^2}-\frac{1}{1+\pax f^2}\right)\\
&&+2\Lambda\pax f\left(\frac{2f\pax f}{(1+l^2-f^2)^2}+\frac{2\pax f\pax^2 f}{(1+\pax f^2)^2}\right)\\
&&+\Lambda f\left(\frac{2\pax f^2+2f\pax^2f}{(1+l^2-f^2)^2}+\frac{2\pax f\pax^3 f+2\pax^2f^2}{(1+\pax f^2)^2}\right)\\
&&+2\Lambda f\left(\frac{(2f\pax f)^2}{(1+l^2-f^2)^3}-\frac{(2\pax f\pax^2 f)^2}{(1+\pax f^2)^3}\right).
\end{eqnarray*}
If we evaluate at $x=\tilde{x}$ and we use the fact that $f(\tilde{x})=l$ is the maximum, we obtain
\begin{eqnarray*}
\pat\pax^2f(\tilde{x})&=&2\pax^2f(\tilde{x})\Lambda f(\tilde{x})\left(l+\pax^2f(\tilde{x})\right).
\end{eqnarray*}
From this ODE we conclude the result.
\end{proof}

\appendix
\section{Auxiliary results}
We provide a bound for the $\Lambda$ acting on the composition of two functions:
\begin{lem}\label{lemaaux}
Given $\Omega=\RR^d,\TT^d$ $F\in C^2(\RR)$ and $h\in W^{\gamma,\infty}(\Omega)$ with $2\gamma>\alpha>0$, we have 
$$
\Lambda^\alpha F(h(x))\leq C(F,\alpha,\gamma,d)\|h\|_{W^{\gamma,\infty}}^2+F'(h(x))\Lambda^\alpha h(x),
$$
where $C(F,\alpha,\gamma,d)=C(\|F\|_{C^2(|z|\leq \|h\|_{L^\infty})},\alpha,\gamma,d)$.
\end{lem}
\begin{proof}We prove this result for $\Omega=\RR^d$. For the torus, the proof follows the same ideas. We define 
$$
W(x,y)=\left\{\begin{array}{cc}
0 &\text{ if }h(x)=h(x-y)\\
\frac{F(h(x))-F(h(x-y))}{h(x)-h(x-y)}-F'(h(x))&\text{ otherwise}.
\end{array}\right.
$$
Notice that, using Taylor Theorem,
$$
|W|\leq C_F|h(x)-h(x-y)|.
$$
Then, given $\epsilon>0$, if $|y|>\epsilon$ we have
\begin{multline*}
\frac{F(h(x))-F(h(x-y))}{|y|^{d+\alpha}}dy=W(x,y)\frac{h(x)-h(x-y)}{|y|^{d+\alpha}}dy\\
+F'(h(x))\frac{h(x)-h(x-y)}{|y|^{d+\alpha}}dy.
\end{multline*}
Consequently,
\begin{eqnarray*}
\int_{\epsilon<|y|<1}\frac{F(h(x))-F(h(x-y))}{|y|^{d+\alpha}}dy&=&\int_{\epsilon<|y|<1}W(x,y)\frac{h(x)-h(x-y)}{|y|^{d+\alpha}}\\
&&+F'(h(x))\frac{h(x)-h(x-y)}{|y|^{d+\alpha}}dy\\
&\leq& C_F\|h\|^2_{\dot{W}^{\gamma,\infty}}\int_0^1\frac{r^{d-1}r^{2\gamma}dr}{|r|^{d+\alpha}}\\
&&+F'(h(x))\int_{\epsilon<|y|<1}\frac{h(x)-h(x-y)}{|y|^{d+\alpha}}dy.
\end{eqnarray*}
For the outer part we have
\begin{multline*}
\int_{1<|y|<\frac{1}{\epsilon}}\frac{F(h(x))-F(h(x-y))}{|y|^{d+\gamma}}dy\leq C_F\|h\|_{L^\infty}^2\\
+F'(h(x))\int_{1<|y|<\frac{1}{\epsilon}}\frac{h(x)-h(x-y)}{|y|^{d+\gamma}}dy.
\end{multline*}
Putting all together and taking the limit $\epsilon\rightarrow0$, we conclude the result.
\end{proof}

We will use a classical compactness result:
\begin{lem}[\cite{Temam}] \label{lemma:2.1}
Let $X_0,X,X_1$ be three Banach spaces such that
$$
X_0\subset X\subset X_1,
$$ 
with continuous embedding and such that $X_i$ are reflexive and the injection $X_0\subset X$ is compact. Let $T>0$ be a finite number and let $\alpha_0,\alpha_1$ be two finite numbers such that $\alpha_i>1$. Then the space
$$
Y=\{u\in L^{\alpha_0}([0,T],X_0),\; \partial_t u\in L^{\alpha_1}([0,T],X_1)\}
$$
is compactly embedded in $L^{\alpha_0}([0,T],X)$.
\end{lem}

\bibliographystyle{abbrv}
%\bibliography{bibliografia}

\end{document}